\documentclass[final]{siamonline1116}
\usepackage{enumitem}
\usepackage{amsfonts}
\ifpdf

  \DeclareGraphicsExtensions{.eps,.pdf,.png,.jpg}
\else
  \DeclareGraphicsExtensions{.eps}
\fi
\usepackage[english]{babel}
\usepackage{amsmath}
\usepackage{amssymb} 
\usepackage[hang,small,bf]{caption}
\usepackage[labelformat=simple]{subcaption}

\usepackage{siunitx}
\usepackage[nocompress,space]{cite}
\numberwithin{theorem}{section}
\usepackage{color}

\graphicspath{{figures/}}
\newcommand{\vare}{\varepsilon}

\newsiamremark{remark}{Remark}  
\crefname{remark}{Remark}{Remarks}  

\makeatletter
\long\def\@makecaption#1#2{%
    \footnotesize
    \setlength{\parindent}{1.5pc}
  \ifx\@captype\@figtxt
    \vskip\abovecaptionskip
    \setbox\@tempboxa\hbox{{\normalfont\color{header1}\HLsmall #1.} {\normalfont\itshape #2}}%
    \ifdim \wd\@tempboxa >\hsize
      {\normalfont\color{header1}\HLsmall #1.} {\normalfont\itshape #2}\par
    \else
      \global\@minipagefalse
      \hb@xt@\hsize{\hfil\box\@tempboxa\hfil}%
    \fi
  \else
    \hbox to\hsize{\hfil{\normalfont\color{header1}\HLsmall #1}\hfil}%
    \setbox\@tempboxa\hbox{{\normalfont\itshape #2}}%
    \ifdim \wd\@tempboxa >\hsize
      {\normalfont\itshape #2}\par
    \else
     \global\@minipagefalse
      \hb@xt@\hsize{\hfil\box\@tempboxa\hfil}%
    \fi
    \vskip\belowcaptionskip
  \fi}
\makeatother

\title{A stiction oscillator under slowly varying forcing: Uncovering small scale phenomena using blowup}
 
\author{Kristian Uldall Kristiansen }

\begin{document}    
%siam_id=M112077
%CODEN=SJADAY
% \slugger{siads}{2017}{16}{4}{2233--2258}
\maketitle

\begin{abstract}
In this paper, we consider a mass-spring-friction oscillator with the friction modelled by a regularized stiction model in the limit where the ratio of the natural spring frequency and the forcing frequency is on the same order of magnitude as the scale associated with the regularized stiction model. The motivation for studying this situation comes from \cite{bossolini2017b} which demonstrated new friction phenomena in this regime. The results of \cite{bossolini2017b} led to some open problems, see also \cite{bossolini2020}, that we resolve in this paper. In particular, using GSPT and blowup \cite{jones1995a,krupa2001a} we provide a simple geometric description of the bifurcation of stick-slip limit cycles through a combination of a canard and a global return mechanism. We also show that this combination leads to a canard-based horseshoe and are therefore able to prove existence of chaos in this fundamental oscillator system.  
\end{abstract}



\begin{DOI}
10.1137/17M1120774
\end{DOI}

% \headers{A horseshoe in a stiction oscillator}{K. U. Kristiansen}
% sds
\section{Introduction}
The importance of slow-fast systems:
\begin{equation}\label{eq:slowfast}
\begin{aligned}
 \dot x &= f(x,y,\varepsilon),\\
 \dot y&=\varepsilon  g(x,y,\varepsilon),
\end{aligned}
\end{equation}
is well recognized in many areas of applied mathematics, perhaps most notably in mathematical neuroscience \cite{izh07,Ter10}. During the last decades these systems have been studied intensively using the framework of Geometric Singular Perturbation Theory (GSPT) \cite{fenichel1974a,fenichel1979a,jones1995a,kuehn2015a}. The point of departure from  this theory, based upon Fenichel's theory of singular perturbations, is the critical set $S_0=\{(x,y)\,:\,f(x,y,0)=0\}$ of \cref{eq:slowfast}$_{\vare=0}$. If $S_0$ is a compact submanifold and the normal hyperbolicity conditions hold, i.e. $D_x f\vert_{S_0}$ only have eigenvalues with nonzero real part, then Fenichel's theory says that $S_0$, as well as its stable and unstable manifolds $W^{s/u}(S_0)$, perturb to $S_{\vare}$ and $W^{s/u}(S_\vare)$, respectively, for all $0<\vare\ll 1$. To deal with points on $S_0$, where the normal hyperbolicity condition is lost, the blowup method \cite{dumortier1996a} has -- following the work of \cite{krupa2001a} -- been used to extend GSPT. This extended theory has been applied in many different scientific contexts to describe complicated dynamics, including relaxation oscillations \cite{kosiuk2009,kosiuk2015a,Kristiansen_2020,kuehn2015b}, as well as canard phenomena due to repelling critical manifolds \cite{carter2017,krupa2001c,szmolyan2001a,szmolyan2004a}.

More recently, systems that limit onto nonsmooth ones as $\vare\rightarrow 0$, e.g.
% 
% 
% GSPT and blowup in particular has been used to describe the dynamics of systems
\begin{equation}\label{eq:nonsm}
\begin{aligned}
 \dot x =f(x,y,\phi(y\vare^{-1})),\\
 \dot y =g(x,y,\phi(y\vare^{-1})),
\end{aligned}
\end{equation}
for $x\in \mathbb R^n$, $y\in \mathbb R$,
with $\phi$ a sigmoidal function $\lim_{s\rightarrow \pm \infty} \phi(s)\rightarrow \pm 1$, have been studied by adapting these methods, see e.g. \cite{kristiansen2018a,kristiansen2019b,kristiansen2020b,llibre2008a}. Systems of the form \cref{eq:nonsm} also occur in many different scientific contexts, for example -- as in electrical engineering and in biological systems -- due to $\phi$ modelling a switch \cite{dibernardo2008a,kristiansen2019b}. More indirectly, systems \cref{eq:nonsm} also occur through regularizations of piecewise smooth systems \cite{kristiansen2014a}, which are common in mechanics. For example, the simplest friction laws are piecewise smooth and systems of the form \cref{eq:nonsm} are therefore common in this area too \cite{bossolini2017b}. 

 As opposed to \cref{eq:slowfast}, the set $\Sigma=\{(x,y)\,:\,y=0\}$ is a discontinuity set (also called a switching manifold) of \cref{eq:nonsm}$_{\vare =0}$. By working in the extended space $(x,y,\vare)$, the reference \cite{kristiansen2018a,kristiansen2020b}, among others, gain smoothness through a cylindrical blowup of points $(x,0,0)$. Using this approach, recent references \cite{jel2020,jel2020b,kristiansen2019b,kristiansen2020b} have described the dynamics and bifurcations of systems like \cref{eq:nonsm} for all $0<\vare\ll 1$.

\subsection{Model}
In this paper, we consider the following model for the spring-mass-friction oscillator illustrated in \Cref{fig:brake_pad}:
\begin{equation}\label{eq:model0}
\begin{aligned}
 \dot x &= y,\\
 \dot y &=-x-\sin \theta-\mu_d \phi(\vare^{-1} y),\\
 \dot \theta &=\omega:=\vare \xi
\end{aligned}
\end{equation}
% as a prototypical system with slow-fast phenomena interacting with nonsmooth effects. The system models the spring-mass-friction oscillator illustrated in \Cref{fig:} for a slowly varying forcing. 
Notice, that the spring (or natural) frequency of the system \cref{eq:model0} has been normalized to $1$ whereas the forcing frequency is $\omega=\vare \xi$; we assume that $\xi \in \mathbb R$ (fixed) and $0<\vare\ll 1$ and the forcing $f(\theta)=\sin \theta$ is therefore ``slowly varying''. At the same time, following the result of \cite{bossolini2017b}, the friction force is modelled as a regularization of the nonsmooth stiction law, through the term $\mu_d \phi(\vare^{-1}y)$, involving the same scale $\vare$, by the following assumptions on the smooth function $\phi$:
% \begin{itemize}
\begin{enumerate}[label=({A}{{\arabic*}})]
 \item \label{it:A1} $\phi(s)\rightarrow \pm 1$ as $s\rightarrow \pm \infty$.
 \item \label{it:A2}$\phi$ is an odd function: $\phi(-s)=- \phi(s)$. 
 \item \label{it:A3} There is one $\delta>0$ and a 
 $$\mu_s>\mu_d,$$
  such that $\phi'(s)>0$ $\Leftrightarrow$ $s\in (-\delta,\delta)$, $\phi'(\delta)=0$ and $\phi(\delta)=\mu_s/\mu_d>1$ with $\phi''(\delta)<0$.
\end{enumerate}

% \ref{it:A1}
The parameters $\mu_d$ and $\mu_s$ are proportional by a nondimensional factor $n$ to the dynamic and static friction coefficients $f_d$ and $f_s$, respectively, and $\mu_s=nf_s>\mu_d=nf_d$ therefore reflects the well-known fact that $f_s>f_d$, which can be interpreted as follows: the force required to keep the mass in motion is smaller than the force required to initiate the motion. The proportionality factor $n$ is a ratio of the normal force to the amplitude of the external forcing, see \cite{bossolini2017b} for further details on the derivation. According to the numbers in \cite{barrett1990a}, $\mu_s/\mu_d=f_s/f_d\sim 1-2$ for steel and others.

\begin{figure}[!t]
        \centering
        \includegraphics[width=0.45\textwidth]{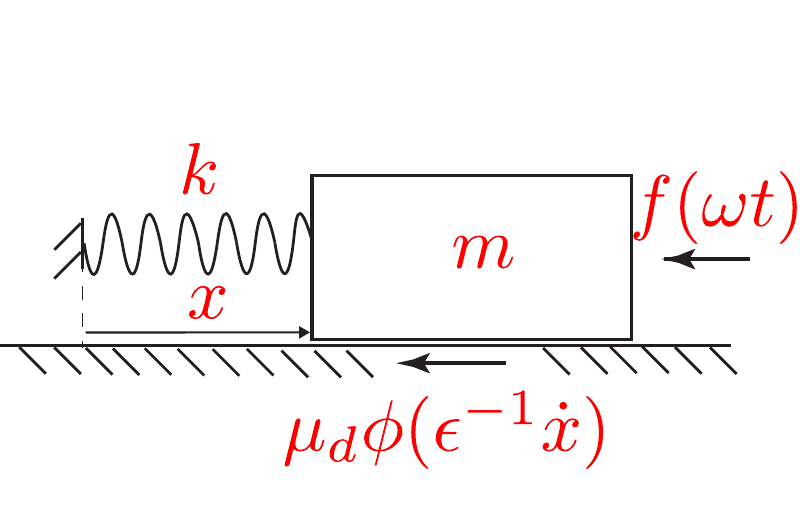}
        \caption{A spring-mass-friction oscillator. We consider a nondimensionalized version where both the mass $m$ and spring constant $k$ are scaled to $1$. We also use a time $t$ nondimensionalized by the natural frequency of the spring. The parameter $\omega$ is then the ratio of the frequency of the forcing $f$ and the natural frequency. In this paper we shall consider $\omega=O(\vare)$ with $\vare$ being a scale associated with the regularized stiction law for friction $\mu_d \phi(\vare^{-1} y)$. }
        \label{fig:brake_pad}
\end{figure} 

% Notice that in combination the conditions on $\phi$ imply that $\phi'(\pm\delta) =0$ and that $\phi'(s)<0$ for $s<-
% \delta$ and $s>\delta$. 
We illustrate $\phi$ in \cref{fig:regul}.
\begin{figure}[!t]
        \centering
        \includegraphics[width=0.45\textwidth]{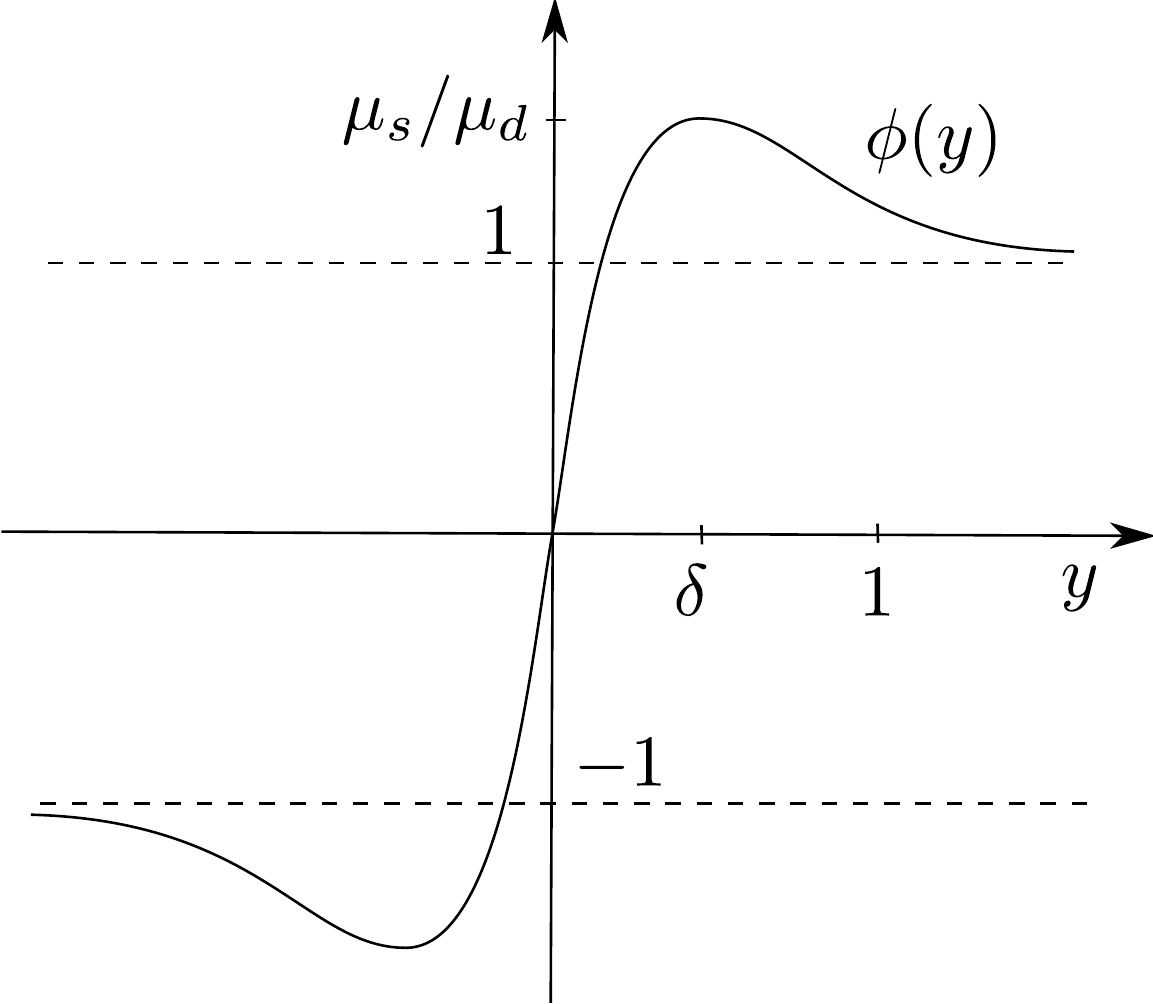}
        \caption{Regularization function satisfying 
\ref{it:A1}, \ref{it:A2} and \ref{it:A3} with $\phi'(\pm \delta)=0$ and where $\phi(y)\rightarrow \pm 1$ as $y\rightarrow \pm \infty$.}
        \label{fig:regul}
\end{figure}
We will also need the following technical assumption on $\phi$:
\begin{enumerate}[resume*]
 \item \label{it:A4} There exists a $k\in \mathbb N$ and an $s_0>0$ such that $\phi(-s^{-1}) = -1-s^k \phi_-(s)$ for all $s\in [0,s_0]$ where $\phi_-:[0,s_0]\rightarrow \mathbb R_+$ is smooth with $\phi_-(0)>0$. 
\end{enumerate}
Notice that the equations \cref{eq:model0} are symmetric with respect to 
\begin{align*}
 \mathbb S:\quad (x,y,\theta)\mapsto (-x,-y,\theta+\pi).
\end{align*}
In fact, most of our results generalize to any forcing $f(\theta)$ with $f(\theta+\pi)=-f(\theta)$ so that the system remains symmetric with respect to $\mathbb S$. We primarily focus on $f(\theta)=\sin \theta$ for simplicity.

System \cref{eq:model0} is a combination of \cref{eq:slowfast} and \cref{eq:nonsm}, in the sense that $\theta$ is slowly varying by the assumption $\omega=O(\vare)$. At the same time, \cref{eq:model0} limits as $\vare\rightarrow 0$ to a simple PWS system:
% These two smooth vector fields have the explicit form
\begin{align} \label{eq:pwsstick}
\begin{cases}
x' &\hspace{-.8pc}= y,\\
y' &\hspace{-.8pc}= -x-\sin \theta \mp \mu_d,\\
\theta' &\hspace{-.8pc}= 0.\end{cases} 
\end{align}
for $y\gtrless 0$, respectively, having $\Sigma: \,y=0$ as a switching manifold. Within $\Sigma$, we find that $x=-\sin \theta-\mu_d$ is a curve of tangencies (fold line) for the $y>0$ system whereas $x=-\sin \theta+\mu_d$ is curve of tangencies (fold line) for the $y<0$ system. Both fold lines (red and blue in \Cref{fig:pws}, respectively) are \textit{invisible} \cite{jef}. Trajectories outside $\Sigma$ are therefore simple arcs of a half circle. These sets are also red and blue, respectively, in \Cref{fig:pws}.

In the context of \cref{eq:model0}, special interest lies in the existence of different limit cycles: An $\varepsilon$-family of limit cycles $\Gamma_\varepsilon$ with $\lim_{\varepsilon\rightarrow 0}\gamma_{\varepsilon}$ well-defined is said to be...
\begin{itemize}
\item \textit{pure-slip} if $\lim_{\varepsilon\rightarrow 0} \Gamma_\varepsilon$ intersects $\Sigma$ in a discrete set of points.
\item \textit{pure-stick} if $\lim_{\varepsilon\rightarrow 0} \Gamma_\varepsilon \subset \Sigma$.
\item \textit{stick-slip} otherwise.
\end{itemize}
Since $\Gamma_\varepsilon$ has to intersect $y=0$ for it to be a periodic orbit, a stick-slip limit cycle enters $y\gtrless \pm c$ for $c>0$ small for all $0<\varepsilon\le \varepsilon_0(c)$ small enough but $\lim_{\varepsilon\rightarrow 0} \Gamma_\varepsilon\cap \Sigma$ is not discrete. 
There will be no pure-slip \cref{eq:model0} for all $0<\vare\ll 1$ and focus will therefore be on pure-stick and stick-slip orbits. 
\begin{figure}[!t]
        \centering
        \includegraphics[width=0.45\textwidth]{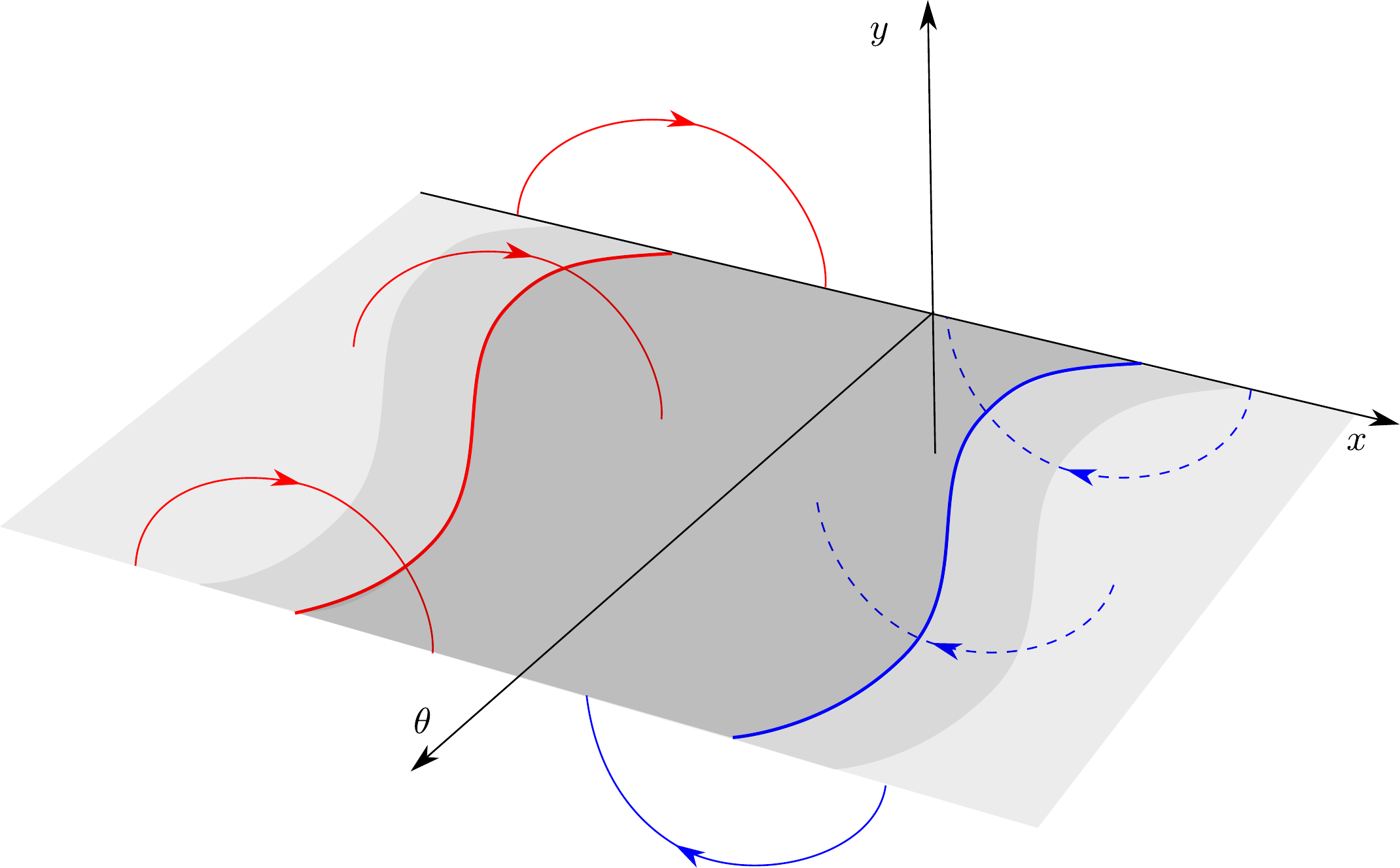}
        \caption{The PWS system. The switching manifold $\Sigma$ has three shades of grey. The darkest band is bounded by the fold lines in blue and red. On either side of this band, there is a different smaller band which is related to the stiction law. In particular, according to this law trajectories ``stick'' to $\Sigma$ until they leave these bands. }
        \label{fig:pws}
\end{figure}

\subsection{Background}

The system \cref{eq:model0} was studied in \cite{bossolini2017b,bossolini2020} for $\omega=\mathcal O(1)$ (using a slightly different scaling) with the main focus on the connection between the smooth system and the PWS system limit with the following rule of stiction on $\Sigma$: If $y(0)=0$, $x(0) +\sin \theta(0)\in (-\mu_s,\mu_s)$ then $y(t)=0$ until $x(t) +\sin \theta(t)\notin (-\mu_s,\mu_s)$. 
The authors defined a physical meaningful notion of solutions at the PWS level (\textit{stiction solutions}) and showed that these solutions could be forward nonunique (\textit{singular stiction solutions}). They also demonstrated that there are special (canard) solutions of the regularized system \cref{eq:model0} for $\omega=\mathcal O(1)$ that produce solutions that are not \textit{stiction solutions \cite[Definition 4.3]{bossolini2017b}}, but yet they appear robustly for any regularization function satisfying \ref{it:A1}-\ref{it:A4}. These canard solutions provide a resolution of the nonuniqueness  at the PWS level. In fact, the dynamics is for $\omega=\mathcal O(1)$ qualitatively independent of $\phi$. On the other hand, \cite{bossolini2017b} also performed numerical computations that showed that there are families of stick-slip periodic orbits which are connected for the smooth system but disconnected for the PWS one. The connectedness of these cycles occur for $\omega$ small enough due to a fold bifurcation, see \Cref{fig:coexistence_bif} (and the caption for further details) which reproduces the results in \cite{bossolini2017b} in this parameter regime. \cite{bossolini2017b} shows that there can be no stick-slip orbits for $\xi< \delta$ for all $0<\vare\ll 1$ but $\xi=\delta$ is only a lower bound for the fold shown in \Cref{fig:coexistence_bif} and it does not explain the main mechanism for the bifurcation. In this paper, we will therefore focus on 
\begin{align}
 \xi>\delta,\label{eq:xi_delta_cond}
 \end{align}
 and describe the limit $\omega=\vare \xi$ as $\vare\rightarrow 0$ using GSPT and blowup. This limit is inaccessible at the PWS level where the orbits within $y=0$ are independent of $\omega$. Our analysis reveal a simple structure that allows for an almost complete description of the long term dynamics in this limit including a simple geometric explanation for the fold bifurcation of limit cycles.  We will see that in this limit, which can only be resolved at the regularized level, the dynamics depend upon $\phi$ at a qualitative level. Interestingly, the bifurcation of limit cycles that we describe is also shown to be associated with the existence of a horseshoe, in a new general (canard-based) mechanism which we lay out below, see \Cref{thm:main2}. We will use $\xi$ as our primary bifurcation diagram and restrict attention to parameters $\mu_s$, $\mu_d$ and $\delta$ (and further details on $\phi$) that are consistent with the behaviour observed in \cite{bossolini2017b,bossolini2020} (we formalize this in the assumption \ref{it:A5} below). Our results resolve all problems on the stiction model that were left open in \cite{bossolini2017b,bossolini2020}.  %For now, however, we describe our approach in further details.

The stiction oscillator \cref{eq:model0} have also been studied at the PWS level in many other references, see e.g. \cite{csernak2006a, hinrichs1998a,licsko2014a}. For example, the references \cite{hinrichs1998a,licsko2014a} perform accurate numerical computations and demonstrate routes to chaotic dynamics through period doubling cascades within $\omega\in (0,1)$. The case when $\mu_s=\mu_d$ is the most classical one. Here a lot more is known analytically about existence of periodic orbits and bifurcations, see e.g. \cite{kowalczyk2008a,guardia2010a,kunze}. This case also lends itself to the theory of Filippov system \cite{filippov1988a,dibernardo2008a}. In \cite{dibernardo}, for example, the authors connect the onset of chaotic dynamics to a PWS grazing sliding bifurcation. Our results complement these existing results, insofar that we describe bifurcation of periodic orbits and new routes to chaotic dynamics for a regularized friction model in a limit that is inaccessible at the nonsmooth level. The results also provide a description of the onset of oscillatory dynamics as we go from $f=\text{const}$ (which corresponds to $\omega=0$) to $f$ periodic.

\subsection{Outline}
The article is organized as follows: In \cref{sec:blowup} we describe our method, based upon blowup and GSPT, to study \cref{eq:model0} for all $0<\varepsilon\ll 1$. Then in \cref{sec:pos} we put this to use and describe simple geometric conditions that ensure existence of periodic orbits of both stick-slip type and pure-stick. In \cref{sec:chaos} these geometric conditions lead us to formulate a new canard-based mechanism for horseshoe chaos. This new mechanism is reminiscent of the chaos in the forced van der Pol \cite{haiduc2009a} insofar that it relates to a folded saddle singularity, but the global geometry and the details are different. %This mechanism also explain the 

\section{The blowup approach}\label{sec:blowup}
To describe the full dynamics of \cref{eq:model0} in the limit $\vare\rightarrow 0$, we follow \cite{kristiansen2018a,kristiansen2020b}: We augment $\vare'=0$ and consider the fast time scale:
\begin{equation}\label{eq:model0_ext}
\begin{aligned}
 \dot x &= \vare y,\\
 \dot y &=\vare\left(-x-\sin \theta-\mu_d \phi(\vare^{-1} y)\right),\\
 \dot \theta &=\vare \omega=\vare^2 \xi,\\
 \dot \vare &=0.
\end{aligned}
\end{equation}
obtained by multiplying the right hand side of \cref{eq:model0} by $\vare$. 
For this system, every point $(x,y,\theta,0)$ is an equilibrium, but $y=0$ is extra singular due to lack of smoothness. We therefore perform a blowup transformation:
\begin{align}\label{eq:blowup0}
 (r,(\bar y,\bar \vare))\mapsto \begin{cases} 
y &=r\bar y,\\
\vare &= r\bar \vare,
                                 \end{cases}
\end{align}
for $(\bar y,\bar \vare)\in S^1$ and $r\ge 0$. In this way, we gain smoothness of the resulting vector-field $\overline X$ on $(x,\theta,r,(\bar y,\bar \vare))\in \mathbb R\times \mathbb T \times [0,\infty)\times S^1$, obtained by pull-back of \cref{eq:model0_ext} through \cref{eq:blowup0}. Moreover, $\bar \vare\ge 0$ is a common factor of $\overline X$ and consequently $\widehat X :=\bar \vare^{-1} \overline X$ will have improved hyperbolicity properties. It is $\widehat X$ that we study in the following. We illustrate the result in \Cref{fig:blowup1}. 

\begin{figure}[!t]
        \centering
        \includegraphics[width=0.75\textwidth]{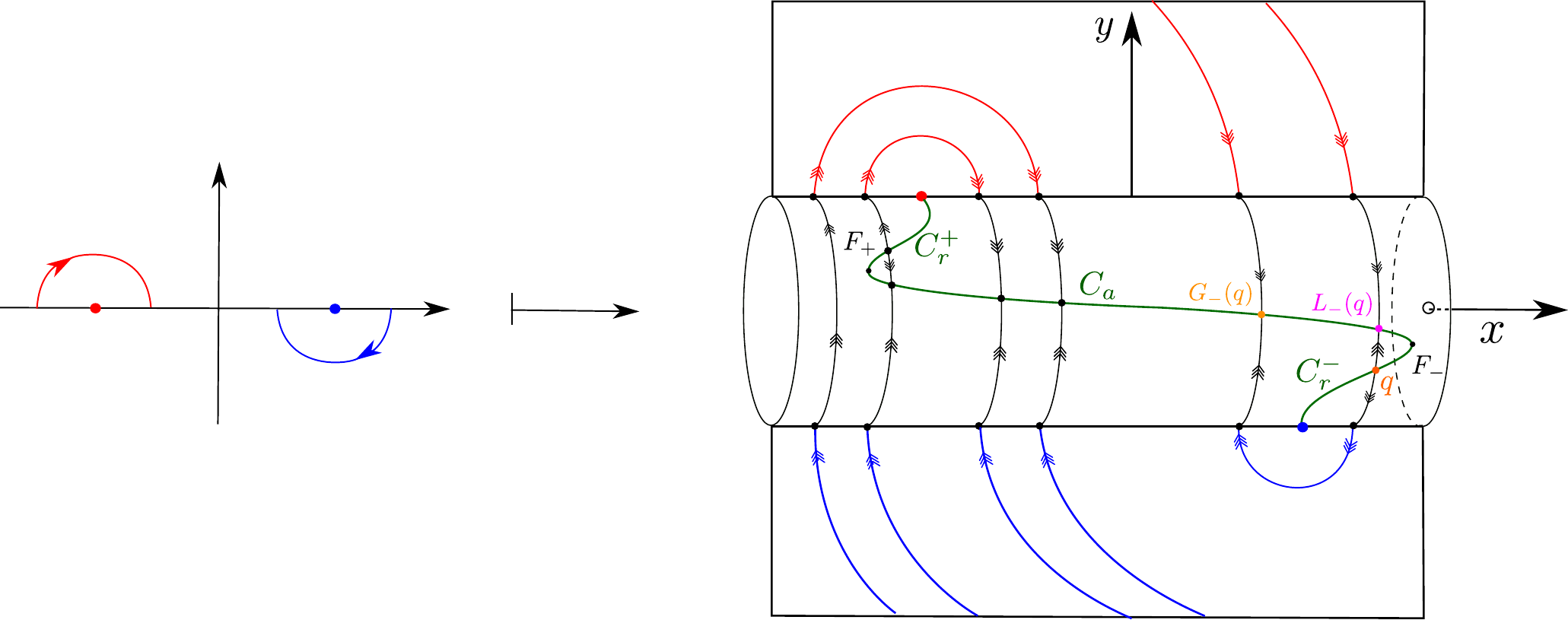}
        \caption{The result of blowing up the degenerate set $\Sigma\times \{0\}$ of \cref{eq:model0_ext}. The figure is illustrated in a section $\theta=\text{const}$. On the cylinder we find a critical manifold $C=C_r^+\cup F_+\cup C_a\cup F_- \cup C_r^-$ in green with two fold lines $F_\pm$ separating the attracting sheet $C_a$ from two repelling ones $C_{r}^\pm$. From the repelling sheets, the blowup approach reveals a return mechanism from $C_r^\pm$ to $C_a$.   }
        \label{fig:blowup1}
\end{figure}

To study $\widehat X$ and perform calculations, we use directional charts. In particular, the scaling chart, obtained by setting $\bar \epsilon=1$
\begin{align*}
 y &=r_2 y_2,\\
 \vare &=r_2,
\end{align*}
with chart-specific coordinates $(r_2,y_2)$, 
produces the following slow-fast equations:
\begin{equation}\label{eq:scaling}
\begin{aligned}
 \dot x &= \vare^2 y_2,\\
 \dot y_2 &=-x-\sin \theta -\mu_d \phi(y_2),\\
 \dot \theta &= \vare^2 \xi,
\end{aligned}
\end{equation}
upon eliminating $r_2=\vare$. Notice that \cref{eq:scaling} is a slow-fast system \cref{eq:slowfast} with
$\vare^2\ge 0$ as the singular perturbation parameter. In slow-fast theory \cite{kuehn2015a}, \cref{eq:scaling} is called the fast system. If $t$ denotes the (fast) time in \cref{eq:scaling} we introduce the slow time $\tau = \vare^2 t$ so that 
\begin{equation}\label{eq:scaling2}
\begin{aligned}
 x' &= y_2,\\
 \vare^2 y_2' &=-x-\sin \theta -\mu_d \phi(y_2),\\
 \theta' &=  \xi,
\end{aligned}
\end{equation}
with respect to this new time. The system \cref{eq:scaling2} is called the slow system.
These systems are obviously topologically equivalent for all $\vare>0$, but the $\vare=0$ limits are not. In particular, setting $\vare=0$ in \cref{eq:scaling} gives the layer problem 
\begin{equation}\label{eq:layer}
\begin{aligned}
 \dot x &= 0,\\
 \dot y_2 &=-x-\sin \theta -\mu_d \phi(y_2),\\
 \dot \theta &= 0,
\end{aligned}
\end{equation}
while $\vare=0$ in \cref{eq:scaling2} gives the reduced problem 
\begin{equation}\label{eq:reduced}
\begin{aligned}
 x' &= y_2,\\
 0 &=-x-\sin \theta -\mu_d \phi(y_2),\\
 \theta' &=  \xi.
\end{aligned}
\end{equation}
In the following, we focus on \cref{eq:layer} and delay the analysis of \cref{eq:reduced} to \cref{sec:reduced}.

For \cref{eq:layer} the set $C$ defined by 
\begin{align}\label{eq:exprC}
 x = -\sin \theta-\mu_d \phi(y_2),\quad \theta\in \mathbb T,\,y_2\in \mathbb R,
\end{align}
is a critical manifold. By linearizing \cref{eq:layer} about any point $(x,y_2,\theta)\in C$ we obtain one single nontrivial eigenvalue:
\begin{align}\label{eq:lambda}
 \lambda(y_2):=-\mu_d \phi'(y_2).
\end{align}
Consequently, the sets $F_\pm := C\cap \{y_2=\pm \delta\}$ are nonhyperbolic fold lines where $\lambda(y_2)=0$. These sets divide $C$ into an attracting sheet:
\begin{align*}
 C_a:=C\cap \{y_2 \in (-\delta,\delta)\},
\end{align*}
where $\lambda(y_2)<0$,
and two repelling sheets:
\begin{align*}
 C_{r}^\pm:=C \cap \{y_2\gtrless \pm \delta\},
\end{align*}
where $\lambda(y_2)>0$.

To connect the analysis of the layer problem with the PWS system, we use the separate charts corresponding to $\bar y=\pm 1$. $\bar y=1$ being identical (in fact we can just apply the symmetry $\mathbb S$) we focus on $\bar y=-1$, where we introduce the chart-specific coordinates $(r_1,\vare_1)$ defined by
\begin{align*}
 y &=- r_1,\\
 \vare &=r_1 \vare_1.
\end{align*}
% \begin{align*}
% y &=r_1,\\
%  \vare &=r_1\vare_1.
% \end{align*}
Notice that we can change coordinates by
\begin{align*}
 y_2 = -\vare_1^{-1}, r_2 = r_1\vare_1.
\end{align*}
Inserting this into \cref{eq:model0} gives
\begin{align*}
 \dot x &= -r_1^2,\\
 \dot r_1 &= r_1 \left(x+\sin \theta+\mu_d (-1-\vare_1^k \phi_-(\vare_1))\right),\\
 \dot \theta &= r_1^2 \vare_1 \xi,\\
 \dot \vare_1 &= -\vare_1\left(x+\sin \theta+\mu_d (-1-\vare_1^k \phi_-(\vare_1))\right),
\end{align*}
where by assumption \ref{it:A4}, we have used that $\phi(\vare_1^{-1})=-1-\vare_1^k \phi_-(\vare_1)$. 
We summarise the properties of this system in \Cref{fig:charty1}. Notice, that for this system, $r_1=\vare_1=0$ is a set of equilibria. In particular, the linearization about any point $(x_0,0,\theta_0,0)$ has only two non-zero eigenvalues whenever $x_0+\sin \theta_0-\mu_d\ne 0$. This gives a saddle structure, as shown in the figure, with heteroclinic connections between points $(x_0,0,\theta_0,0)$ with $x_0+\sin \theta_0-\mu_d>0$ and $(x_1,0,\theta_0,0)$ with 
\begin{align*}
 x_1 &= 2\mu_d-2\sin \theta_0-x_0.
\end{align*}
This follows from setting $\vare_1=0$ and dividing out the common factor $r_1$: 
\begin{figure}[!t]
        \centering
        \includegraphics[width=0.65\textwidth]{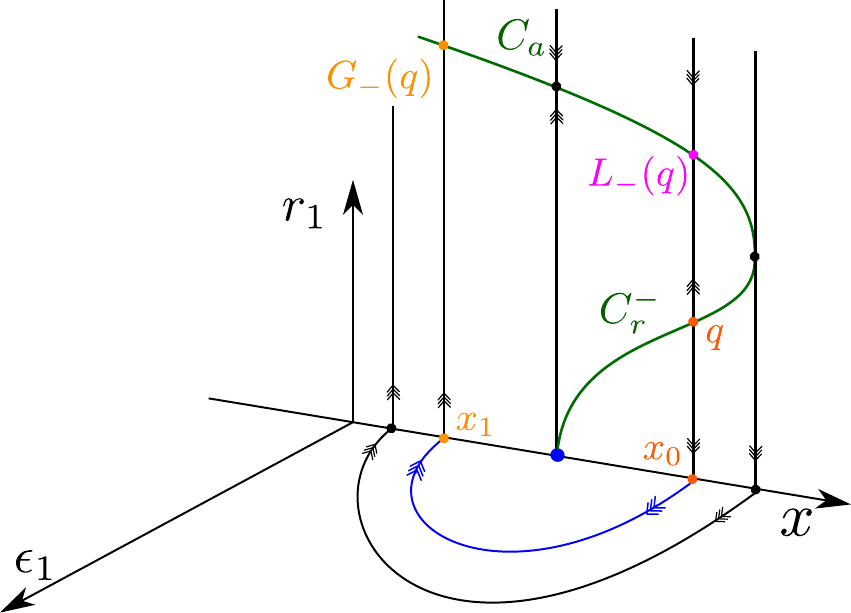}
        \caption{Dynamics in the chart $\bar y=-1$.  }
        \label{fig:charty1}
\end{figure}
% 
%  In particular, any point $(x_0,0,\theta_0,0)$ with $x_0+\sin \theta_0-\mu_d>0$ has an unstable manifold, contained within $\vare_1=0$, $\theta=\theta_0$, and a stable manifold contained within $x=x_0,\,r_1=0$, $\theta=\theta_0$. On the other hand, any point $(x_0,0,\theta_0,0)$ with $x_0+\sin \theta_0-\mu_d<0$ has a stable manifold, contained within $\vare_1=0$, $\theta=\theta_0$, and a stable manifold contained within $x=x_0,\,r_1=0$, $\theta=\theta_0$
% 
% which is given by the orbit segment 
% \begin{align*}
%  x(s) = 
% \end{align*}
% 
% 
% of 
\begin{align*}
 x ' &=-r_1,\\
 r_1' &= x+\sin \theta_0-\mu_d,
\end{align*}
which is just \cref{eq:pwsstick} within ${\theta=\theta_0}$ and upon setting $r_1=-y$. In particular, the forward orbit of $x(0)=x_0,r_1(0)=0$ returns to $r_1(T)=0$ after a time $T=\pi$ with $x_1:=x(T) = 2\mu_d-2\sin \theta_0 -x_0$.
\begin{remark}
Notice that the degenerate point $x=-\sin \theta_0+\mu_d$ is a fixed point of the associated mapping $x_0\mapsto x_1$. This point within $r_1=\vare_1=0$ can be blown up to a sphere (using \ref{it:A4}) in such a way that hyperbolicity is gained. We skip the details of such a blowup analysis, since it is not important for our purposes, and just illustrate the result in \Cref{fig:blowup2}. 
\end{remark}
Following this analysis, we use the hyperbolic structure to track orbits for $0<\vare\ll 1$ away from the critical manifold and deduce two things:
% \begin{lemma}
(1) Each point away from $C$ eventually reaches a neighborhood of $C_a$ for all $0<\vare \ll 1$. (2) Moreover, by patching together the analysis in $\bar y=\pm 1$ and $\bar \vare=1$ we obtain two maps from each 
% 
% 
% the unstable manifold $x=x_0,r_1=0,\theta=\theta_0,\vare_1\ge 0$ of the point $(x_1,0,\theta_0,0)$ with $x_0 > -\sin \theta_0+\mu_d$ into the scaling chart we obtain a return mechanism from the repelling sheet of 
repelling sheets $C_r^\pm $ to the attracting one $C_a$. We will refer to these maps as return maps since orbits will follow these upon leaving $C_a$. In particular, for $C_r^-$ we obtain one such a return mechanism by following the singular flow $y_2$-negative side of $C_r^-$; we describe this by a map $G_-$ from $C_r^-$ to $C_a$ given by
\begin{align*}
 G_-:(x_0,y_{20},\theta_0)\mapsto (x_1,y_{21},\theta_0),
\end{align*}
where $x_1= 2\mu_d-2\sin \theta_0-x_0$ as above and where
$y_{21}$ is the least positive solution of 
\begin{align*}
 x_1 = -\sin \theta_0-\mu_d \phi(y_{21}).
\end{align*}
We can easily rewrite this equation as
\begin{align}\label{eq:y21}
\phi(y_{21}) = -2-\phi(y_{20}),
\end{align}
from which it clearly follows that $G_-$ is well-defined.
We have.
\begin{lemma}
 $G_-$ is well-defined for all $(x_0,y_{20},\theta_0)\in C_r^-$ corresponding to points on the repelling sheet, i.e. 
 $x_0\in (-\sin \theta_0+\mu_d,-\sin \theta+\mu_s)$, $y_{20}<-\delta$. Furthermore, $G_-:C_r^-\rightarrow C_a$ is a diffeomorphism onto its image and the image value of $y_2$ satisfies $(-\delta,\delta)$.
\end{lemma}
\begin{proof}
 Simple calculation, see also \Cref{fig:charty1}. In particular, by \cref{eq:y21} and since $\phi(y_{20})\in (-1,-\mu_s/\mu_d)$, we have that 
 \begin{align*}
  \phi(y_{21}) \in (-1,-2+\mu_s/\mu_d)\in (-1,\mu_s/\mu_d),
 \end{align*}
using that $\mu_s/\mu_d>1$ in the first inclusion.
\end{proof}
% It is a simple calculation to show that such $y_{21}$ always exists for 
It follows that the $y_2$-negative side of the unstable manifold of $C_r^-$ is foliated by stable fibers with base points on $C_a$, according to the assignment $G_-$ and vice versa.
The map $G_+$ from $C_r^+$ to $C_a$ is given $G_+:=\mathbb S\circ G_- \circ \mathbb S$ using the symmetry $\mathbb S$. 

\begin{figure}[!t]
        \centering
        \includegraphics[width=0.75\textwidth]{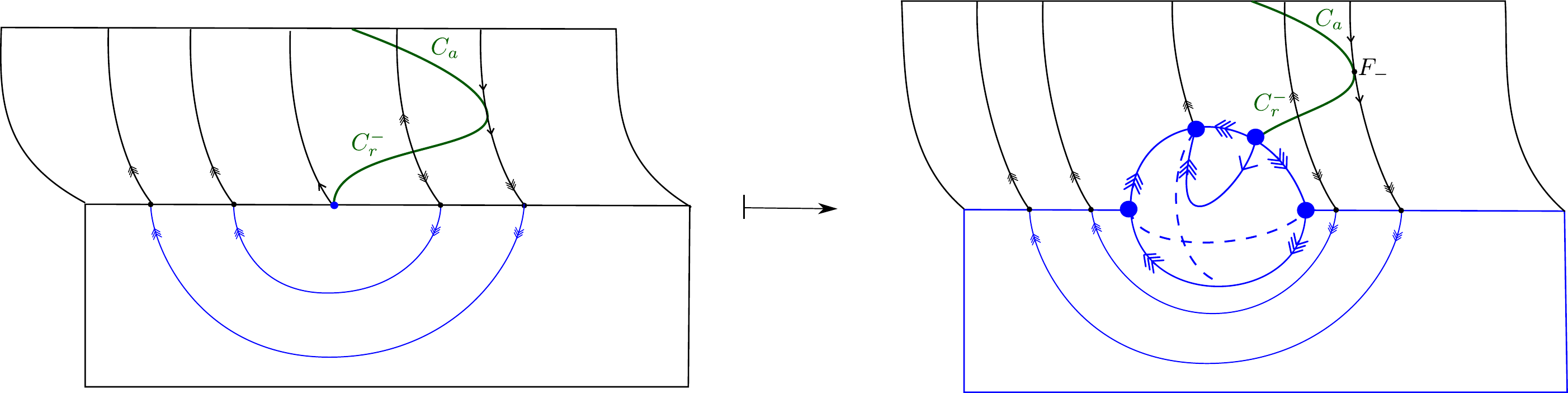}
        \caption{The result of blowing up a degenerate point $x=-\sin \theta +\mu_d$.   }
        \label{fig:blowup2}
\end{figure}

$G_-$ extends to the fold line $F_-$. Here the image is contained within $\{y_2=y_{2-}\}$ with $y_{2-}$ defined by 
\begin{align}\label{eq:y2m}
\phi(y_{2-})=-2+\mu_s/\mu_d,
\end{align}using \cref{eq:y21} and \ref{it:A3}. 

Clearly, as illustrated in \Cref{fig:charty1} there is also a simpler return mechanism from $C_r^-$, which can be described in the chart $\bar \vare=1$ only, where the flow follows the $y_2$-positive side of the unstable manifold. This mapping $L_-$ from $C_r^-$ to $C_a$ is given by 
\begin{align*}
 (x_0,y_{20},\theta_0)\mapsto (x_0,y_{21},\theta_0),
\end{align*}
with $(x_0,y_{20},\theta_0)\in C_r^-$ and where $y_{21}$ is the unique solution of 
\begin{align*}
 \phi(y_{21}) = \phi(y_{20}),\quad y_{21}>-\delta.
\end{align*}
Again, the mapping $L_+$ from $C_r^+$ to $C_a$ is given by the symmetry as $L_+=\mathbb S\circ L_-\circ \mathbb S$. 

% We skip the details, which can be found in \cite{} in similar settings, and just present the result in \Cref{fig:}. In this figure, we have also blown up the degenerate fold lines. 

\subsection{The reduced problem}\label{sec:reduced}
Next, we describe the reduced problem on $C$ working in the scaling chart $\bar \vare=1$, recall \cref{eq:reduced}. 
% \begin{align*}
%  x' &=y_2,\\
%  \theta' &= \xi,
% \end{align*}
Seeing that $C$ is naturally parameterized by $y_2$ and $\theta$, we differentiate \cref{eq:exprC} with respect to the fast time and rewrite \cref{eq:reduced} as
\begin{equation}\label{eq:reducedyt}
\begin{aligned}
 \mu_d \phi'(y_2) y_2' &=-y_2 -\xi \cos \theta,\\
 \theta'&=\xi,
\end{aligned}
\end{equation}
on \cref{eq:exprC}. The fold lines $F_\pm$ on $C$ where $\phi'(y_2)=0$ are singular for \cref{eq:reducedyt}.
We study these equations in the classical way \cite{szmolyan2001a} using desingularization through multiplication of the right hand side by $\xi^{-1} \mu_d \phi'(y_2)$. This gives
\begin{equation}\label{eq:reddes}
\begin{aligned}
y_2' &=-\xi^{-1} y_2 -\cos \theta,\\
 \theta'&= \mu_d \phi'(y_2).
\end{aligned}
\end{equation}
These systems are topologically equivalent on $C_a$ and topologically equivalent upon time-reversal on $C_r^\pm$. The equilibria of \cref{eq:reddes} are folded singularities \cite{szmolyan2001a} and given by $(y_2,\theta)=(-\delta,\theta_+(\xi^{-1}))$, $(y_2,\theta)=(\delta,\theta_-(\xi^{-1}))$ and
\begin{align*}%\label{eq:zpm}
 z_-:\,(y_2,\theta) = (-\delta,\theta_-(\xi^{-1})),\quad z_+:\,(y_2,\theta) = (\delta,\theta_-(\xi^{-1})),
\end{align*}
where
\[\theta_\pm(\xi^{-1}):= \cos^{-1}  (\mp \delta\xi^{-1}).\]
At $\xi=\delta$ there is a fold bifurcation of folded singularities, but as mentioned in the introduction, see \cref{eq:xi_delta_cond}, we restrict attention to $\xi>\delta$ where there are four folded singularities. Two of these folded singularities $z_\pm \in F_\pm$, related by the symmetry $\mathbb S$, are of saddle-type whereas the two remaining ones are either folded foci or folded nodes, depending on $\xi$. In particular, a simple calculation shows that there is a $\xi_{dn}>\delta$ such that these latter folded singularities are folded foci for all $\xi>\xi_{dn}$; this is the case shown in \Cref{fig:reduced}. %We illustrate the reduced problem in \Cref{fig:reduced} under the assumption \cref{eq:xi_delta_cond}.  

The fold lines $F_\pm$ consist of the folded singularities and regular fold points. Let $J_-$ be the set of regular jump points on the fold line $F_-$ where $\xi^{-1}\delta- \cos \theta<0$, so that $y_2$ is decreasing for the reduced flow. We define $J_+$ in a similar way. In fact, it is given by $J_+=\mathbb SJ_-$. For these jump points, $G_\pm$ give a return mechanism to $C_a$, which we illustrate in blue in \Cref{fig:reduced}.

Consider the folded saddle $z_-$ on $F_-$. This point, being a hyperbolic saddle for \cref{eq:reddes}, produces a (singular) vrai canard $\gamma_-$  as a stable manifold $W^s(z_-)$ for \cref{eq:reddes}. This canard connects $C_a$ with $C_r^-$ in finite forward time for \cref{eq:reduced}. Let $\tilde \gamma_-$ denote the subset of $\gamma_-$ on $C_r^-$. For these sets of points, $G_-$ and $L_-$ produce two different return mechanisms to $C_a$ by following the associated unstable fibers on either side of $C_r^-$. These set of points are illustrated in \Cref{fig:reduced} along with their symmetric images for the set of points on $\tilde \gamma_+$ with $\gamma_+=\mathbb S\gamma_-$ being the vrai canard of the folded saddle $z_+$ on $F_+$ \cite{szmolyan2001a}.

\begin{figure}[!t]
%         \centering
%         \includegraphics[width=0.75\textwidth]{reduced.pdf}
     \centering
        \begin{subfigure}{.49\textwidth}\caption{ }\label{fig:reducedPWS}
  \centering
\includegraphics[scale=0.415]{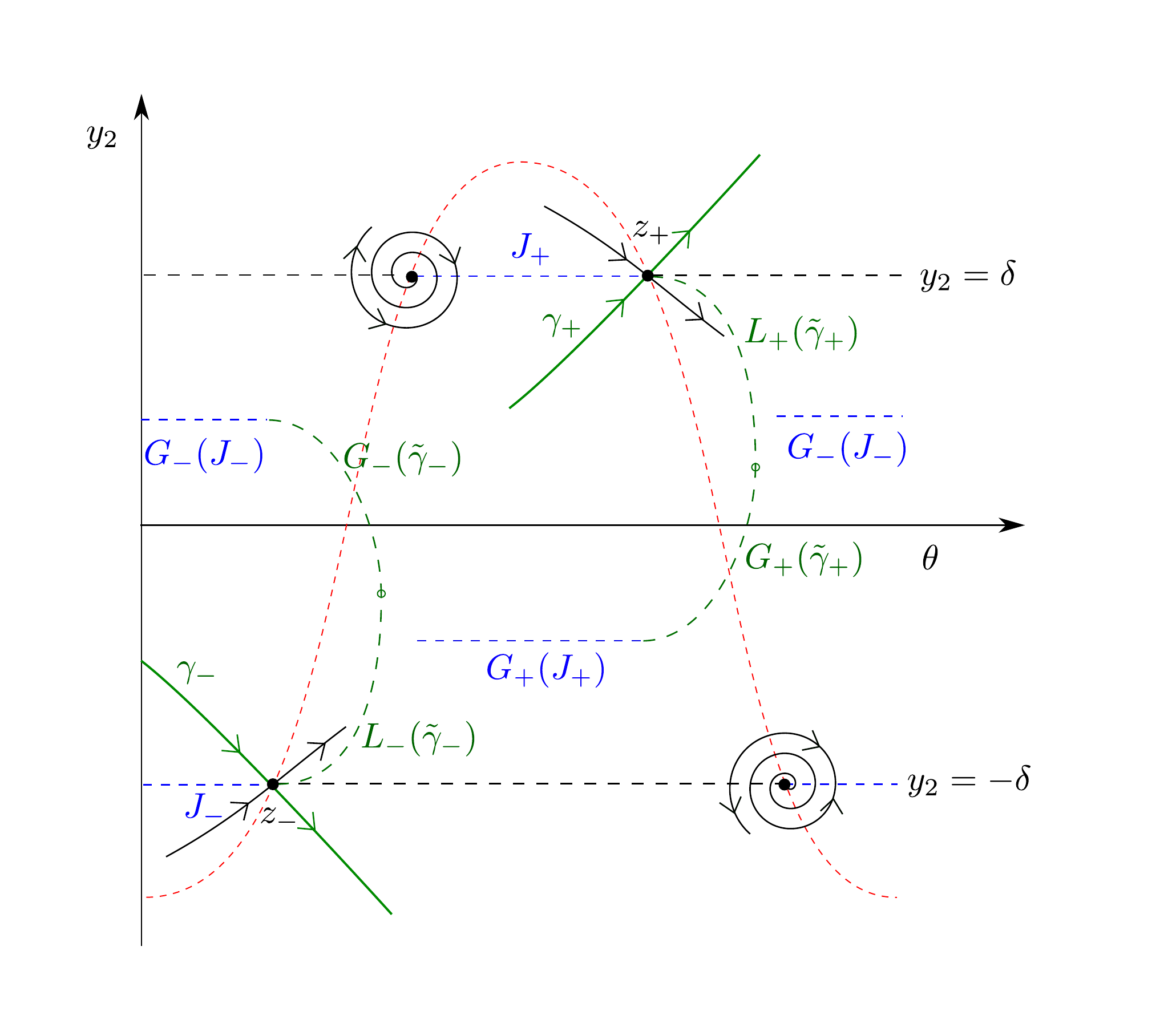}\end{subfigure}
\begin{subfigure}{.49\textwidth}\caption{ }\label{fig:reducedxyt}
  \centering
 \includegraphics[scale=0.415]{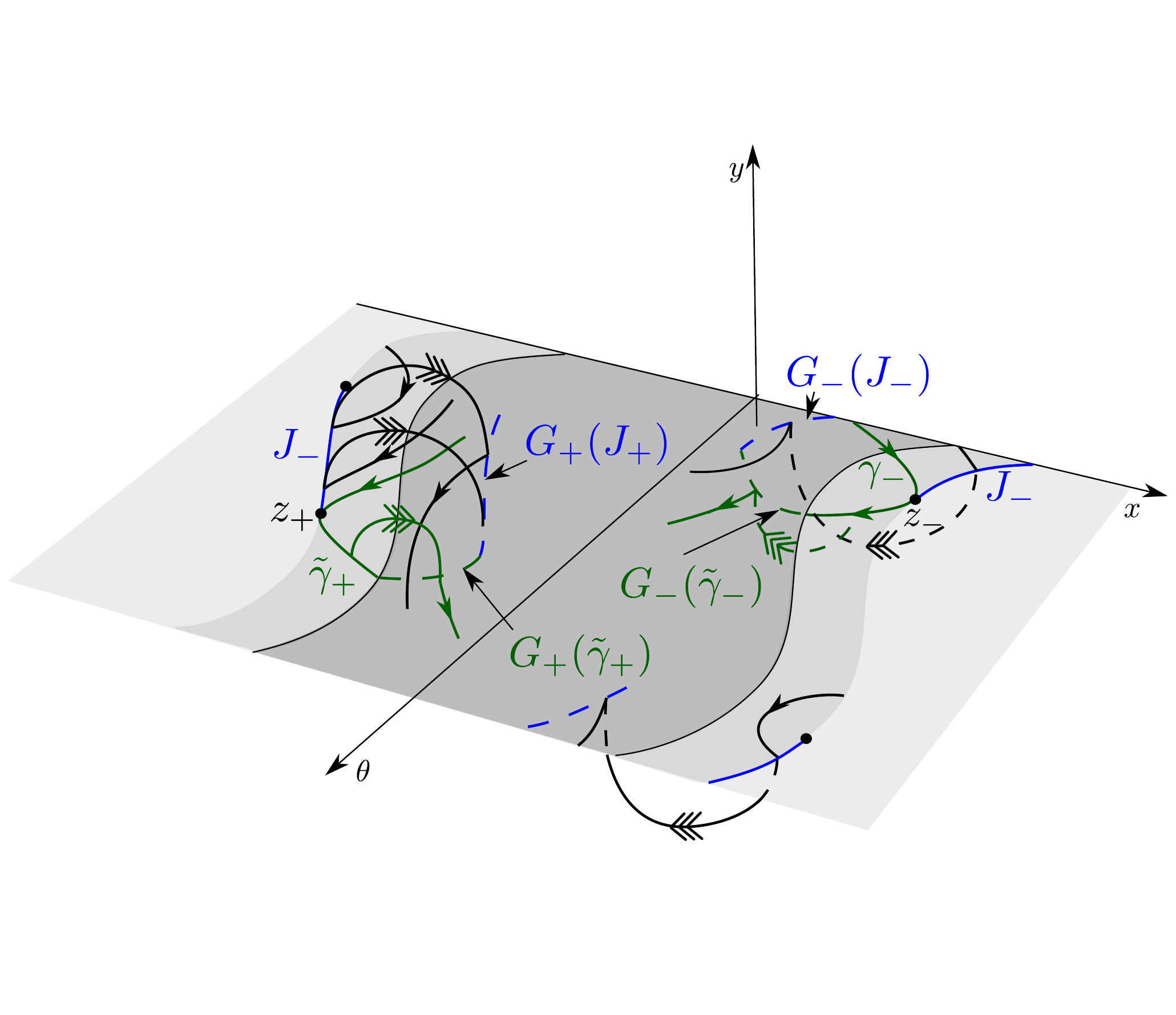}\end{subfigure}
        \caption{In (a): the reduced problem. For $\xi>\delta$ there are four folded singularities, two of which $z_\pm$ are folded saddles whereas the two remaining ones are either folded foci or folded nodes, both stable for the desingularized system \cref{eq:reddes}. In (b): the reduced problem upon blowing down to the $(x,y,\theta)$-space within $\vare=0$ together with the returns obtained by the PWS flows. The flow within $y=0$ is slow in the sense that it is obtained  on a separate (slow) time scale. Orbits within $y\gtrless 0$ are therefore fast in comparison and are therefore given tripple headed arrows whereas in contrast orbits on $y=0$ have a single arrow. Consequently, the limit cycles we obtain for $0<\vare\ll 1$ are generalized relaxation oscillations. Notice also that in this blow down version in (b) we cannot distinguish $C_a$ from $C_{r}^\pm$ on the out-most bands, recall the discussion of these in \Cref{fig:pws}. }
        \label{fig:reduced}
\end{figure}

\subsection{The dependency on $\mu_d$ and $\mu_s$}
The global dynamics of the system depend on the relative position of $\gamma_\pm$, $G_\pm(J_\pm)$, $G_\pm(\tilde \gamma_\pm)$ and $L_\pm(\tilde \gamma_\pm)$ on $C_a$. We can describe $\gamma_\pm$ in the limit $\xi\rightarrow \infty$ of \cref{eq:reduced} since the system is Hamiltonian there with $H(\theta,y_2) = \mu_d \phi(y_2)+\sin \theta$. In particular, a simple calculation shows that for $\mu_s>1$, the stable and unstable manifolds $\gamma_\mp$ of $z_\mp$ with $\theta_\mp(0) =\pi/2,3\pi/2$, respectively, coincide on $C_a$ in this limit given that $\theta\in \mathbb T$. (There is a heteroclinic bifurcation at $\mu_s=1$ where $\gamma_\pm$ connects $z_\pm$.) Moreoever, upon using that the minimum and maximum values of $y_2$ along $\gamma_\pm$ occur at $\theta=\pi/2$, $\theta=3\pi/2$ for $\xi\rightarrow \infty$, it is also straightforward to show that these manifolds remain within $y_2\gtrless 0$, respectively, whenever $\mu_s>2$. Next, using $H$ we can also show that the stable manifold $\gamma_+$ of $z_-$ in this limit intersects $G_-(J_-)$ if and only if $\mu_d<1$, doing so transversally in the affirmative case. (To prove this we simply compare the minimum $y_2$-value of $\gamma_+$ at $\theta=\pi/2$, given by $\mu_d\phi(y_2)=\mu_s-2$ using $H$, with the value of $y_{2-}$, recall \cref{eq:y2m}.)

Consider the $(y_2,\theta)$-plane and let $\Upsilon$ denote the section at $\{y_2=y_{2-}\}$ with $$\theta\in (-\cos^{-1}(-\xi^{-1} y_{2-}),\cos^{-1}(-\xi^{-1} y_{2-})),$$ such that $\dot y_2<0$ on $\Upsilon$, see \cref{eq:reddes}. Hence $G_-(J_-)\subset \Upsilon$ and these sets coincide when $\xi\rightarrow \infty$. Now, suppose $\mu_s>2$ and $\mu_d<1$. Then upon combining the previous results, we have that $\gamma_+\subset \{y_2>0\}$ and that $\gamma_+$ transversally intersects $\Upsilon$ in the limit $\xi\rightarrow \infty$. We now continue this intersection point $(y_{2-},\theta_{\Upsilon}(\xi^{-1}))$ of $\gamma_+$ and $\Upsilon$ for $\xi^{-1}>0$ small enough. Here $\theta_\Upsilon(\xi^{-1})$ is smooth by the implicit function theorem. In fact, a simple Melnikov calculation -- using the fact that the integrand $$(\dot y_2,\dot \theta)\wedge \partial_{\xi^{-1}} (\dot y_2,\dot \theta)\vert_{\xi^{-1}=0}=y_2\dot \theta>0,$$ of the Melnikov integral has one sign by our assumption on $\mu_s>2$ -- shows that $\theta_\Upsilon'(0)>0$. Moreover, using the same argument, $\theta_\Upsilon'(\xi^{-1})>0$ for any $\xi>0$ for which the intersection of $\gamma_+$ with $\Upsilon$ exists. Consequently, there is a unique $\xi_{pd}>0$ such that $\theta_\Upsilon(\xi_{pd}^{-1}) = \theta_-(\xi_{pd}^{-1})$, i.e. where $\gamma_+$ intersects $\Upsilon$ in the point $G_-(z_-)$, see \Cref{fig:xipdph} for an illustration of this situation. (Notice in particular that the $\theta$-coordinate of $z_-$, $\theta_-(\xi^{-1})$, moves in the opposite direction since $\theta_\Upsilon'(\xi^{-1})<0$). Since $y_2$ remains negative on $\tilde \gamma_-$ on the repelling side of $C$, the Melnikov-calculation also provides a monotonicity condition on the set $G(\tilde \gamma_-)$, which in turn for $\mu_d<1$ and $\mu_s>2$ implies that there is also a unique $\xi_{t}>\xi_{pd}$, such that $\gamma_+$ for $\xi=\xi_{t}$ intersects $\Upsilon$ transversally while being tangent to the set $G_-(\tilde \gamma_-)$ at a point $G_-(\tilde \gamma_-)\cap \gamma_+$. %We illustrate the results in \Cref{fig:}.    

In the following, we consider all $\phi$-functions, specifically all $\delta>0$, $\mu_d<1$ and all $\mu_s>1$, such that $\theta_\Upsilon(\xi^{-1})$, defined as the $\theta$-coordinate of the intersection of $\gamma_+$ with $\Upsilon$, which is well-defined for all $\xi^{-1}$ small enough, satisfies these properties:
\begin{enumerate}[resume*]
 \item \label{it:A5} There exist (a) a $\xi_0>0$ such that $\theta_\Upsilon(s)$ is well-defined for all $s\in [0,\xi_0^{-1})$ and satisfies $\theta_\Upsilon'(s)>0$, (b) a $\xi_{pd}<\xi_0$ such that $\theta_\Upsilon(\xi_{pd}^{-1})=\theta_-(\xi_{pd}^{-1})$ and finally (c) a $\xi_t<\xi_{pd}$ such that $\gamma_+$ for $\xi=\xi_{t}$ is tangent to $G_-(\tilde \gamma_-)$ a single point $G_-(\tilde \gamma_-)\cap \gamma_+$,  see \Cref{fig:xitph} for an illustration of this situation.
\end{enumerate}
% \begin{itemize}
%  \item[\ref{it:A5}]
% \end{itemize}
From the preceding discussion we have (we leave out further details for simplicity):
\begin{lemma}
 $\mu_s>2$ and $\mu_d<1$ are sufficient conditions for \ref{it:A5} to hold for any $\delta>0$ (and any further details of $\phi$).
\end{lemma}
% From a mechanical point of view these parameter values may not be realistic. 
But the conditions in this lemma are (clearly) not necessary; it is even clear by a continuity argument that we for any $\mu_d<1$ can take $\mu_s>2-c$ for $c>0$ small enough. In \Cref{fig:thetaUps}, we numerically verify that \ref{it:A5} holds for the regularization function 
\begin{align}
 \phi(s) = \frac{x}{\sqrt{x^2+1}}\left(1+\frac{\alpha }{1+\beta x^2}\right),\label{eq:phie}
\end{align}
for which $k=2$, recall the assumption \ref{it:A4}, and where $\alpha$ and $\beta$ are given by the complicated expressions
\begin{align*}
\alpha &=  \delta^{2} \left( 2\delta^{2}+1 \right) -2{\frac {\delta^{3}
\sqrt {\delta^{2}+1}}{\mu}}
,\quad 
\beta = 2\delta \mu \left( \delta^{2}+1 \right)^{3/2}-4\delta^{2}
 \left(\delta^{2}+1 \right) +2{\frac {\delta^{3}\sqrt {\delta
^{2}+1}}{\mu}}
,
\end{align*}
ensuring that $\phi'(\pm \delta)=0$ and $\phi(\pm \delta) = \pm \mu$ with $\mu:=\mu_s/\mu_d$. This holds for any $\delta>0$ and any $\mu>1$. More specifically, in \Cref{fig:thetaUps}, we fix $\delta =0.6$, $\mu_s=1.1$ and for $10$ equidistributed values of $\mu_d$ between $0.4$ and $0.925$ we visualize the graph of the corresponding $\theta_\Upsilon$ (which we compute numerically  in Matlab using ODE45 and a simple shooting method).  We see that $\theta_\Upsilon$ is monotone and that there are two $\xi$-values, $\xi_{pd}$ (circles) and $\xi_t$ (squares) with the properties in \ref{it:A5} for each $\mu_d\in [0.4,0.925]$. See figure caption for further details. In \Cref{fig:xipdph}, \Cref{fig:xitpdph} and \Cref{fig:xitph} we show the phase portraits (computed using Matlab's ODE45) for
%  $\alpha=7.933$ and $\beta=2.2936$.
\begin{align}
 \delta = 0.6,\,\mu_s=1.1,\,\mu_d = 0.4,\label{eq:parahere}
\end{align}
which are the values used in \cite{bossolini2017b}, and for $\xi=\xi_{pd}\approx 0.8179 $, $\xi=0.7950$ and $\xi=\xi_t\approx 0.7835$, respectively. In these figures the singular vrai canards $\gamma_\pm$ are in green, the faux canards are in black, the sets $G_\pm (J_\pm)$ in blue and finally $G_\pm (\tilde \gamma_\pm)$ and $L_\pm (\tilde \gamma_\pm)$ are all in green and dashed.

Notice that $\gamma_+$ in green passes through $G_-(z_-)$ for $\xi=\xi_{pd}$ whereas it is tangent to $G_-(\tilde \gamma_-)$ for $\xi=\xi_t$, see \Cref{fig:xipdph} and \Cref{fig:xitph}. For $\xi=0.7950$ inbetween these values there are two transverse intersection points (indicated using green circles) of $\gamma_+$ and $G_-(\tilde \gamma_-)$, see \Cref{fig:xitpdph}. 
% \begin{align}
%  \phi(s) = \frac{x}{\sqrt{x^2+1}}\left(1+\frac{\alpha }{1+\beta x^2}\right),\label{eq:phie}
% \end{align}
% for which $k=2$, recall the assumption \ref{it:A4}, with
%  $\alpha=7.933$ and $\beta=2.2936$. \Cref{fig:xipdph} and \Cref{fig:xitph} show the phase portraits for $\xi=\xi_{pd}$ and $\xi=\xi_{t}$, respectively. In fact, \Cref{fig:thetaUps} shows that \ref{it:A5} holds for the regularizaton function in \cref{eq:phie} for any $\mu_d \in [0.4,0.95]$. 
Additional numerical experiments (not shown) lead us to speculate that \ref{it:A5} holds for any $\mu_d<1$ $\mu_s>1$, but we have not find a way of proving this. (In particular, the only reason to restrict to $\mu_d\in [0.4,0.925]$ in \Cref{fig:thetaUps} was that it gave nicer plots.)

\begin{figure}[t!]
        \centering
        \begin{subfigure}{.49\textwidth}\caption{ }\label{fig:thetaUps}
  \centering
\includegraphics[scale=0.415]{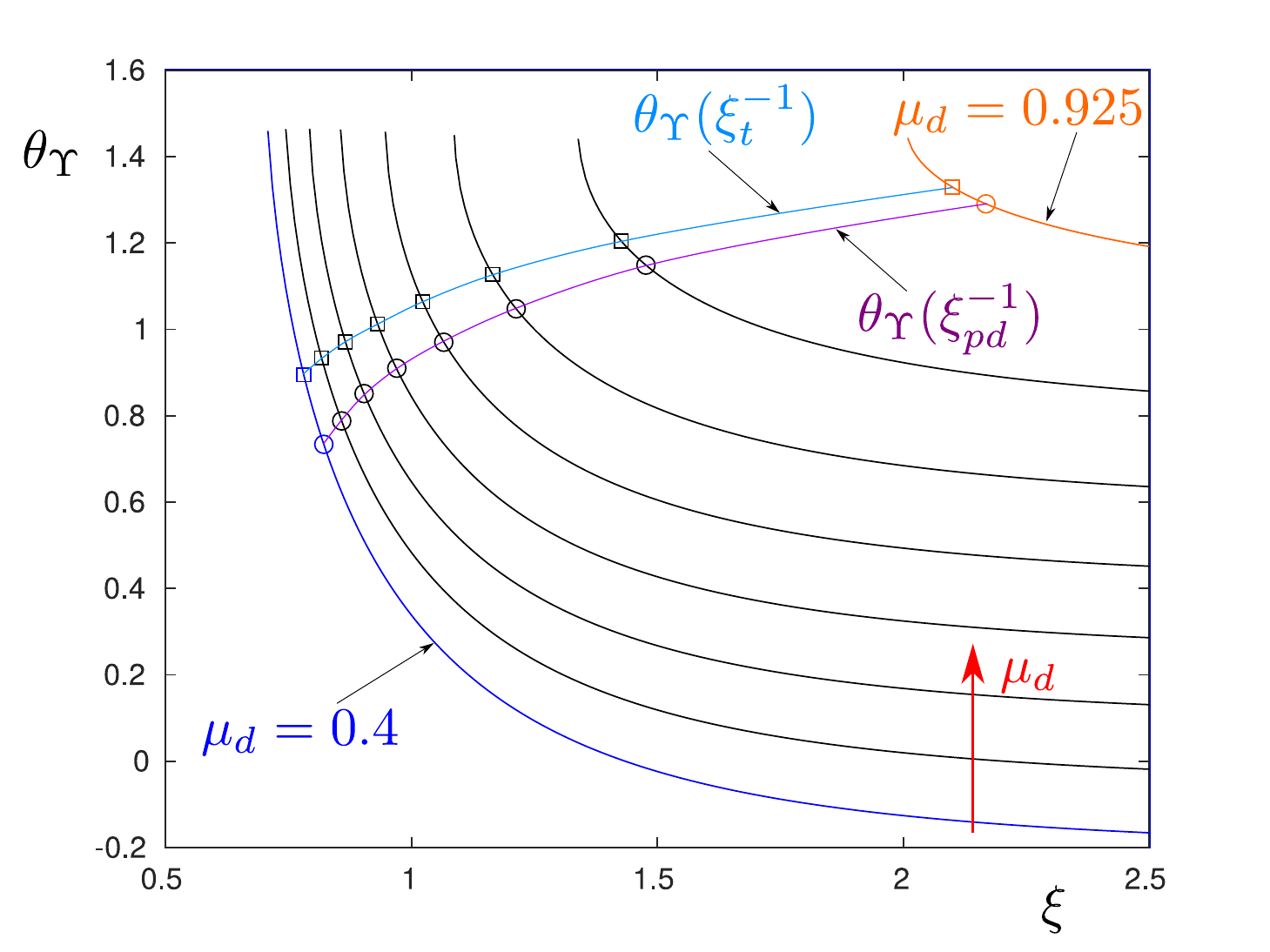}\end{subfigure}
\begin{subfigure}{.49\textwidth}\caption{ }\label{fig:xipdph}
  \centering
 \includegraphics[scale=0.415]{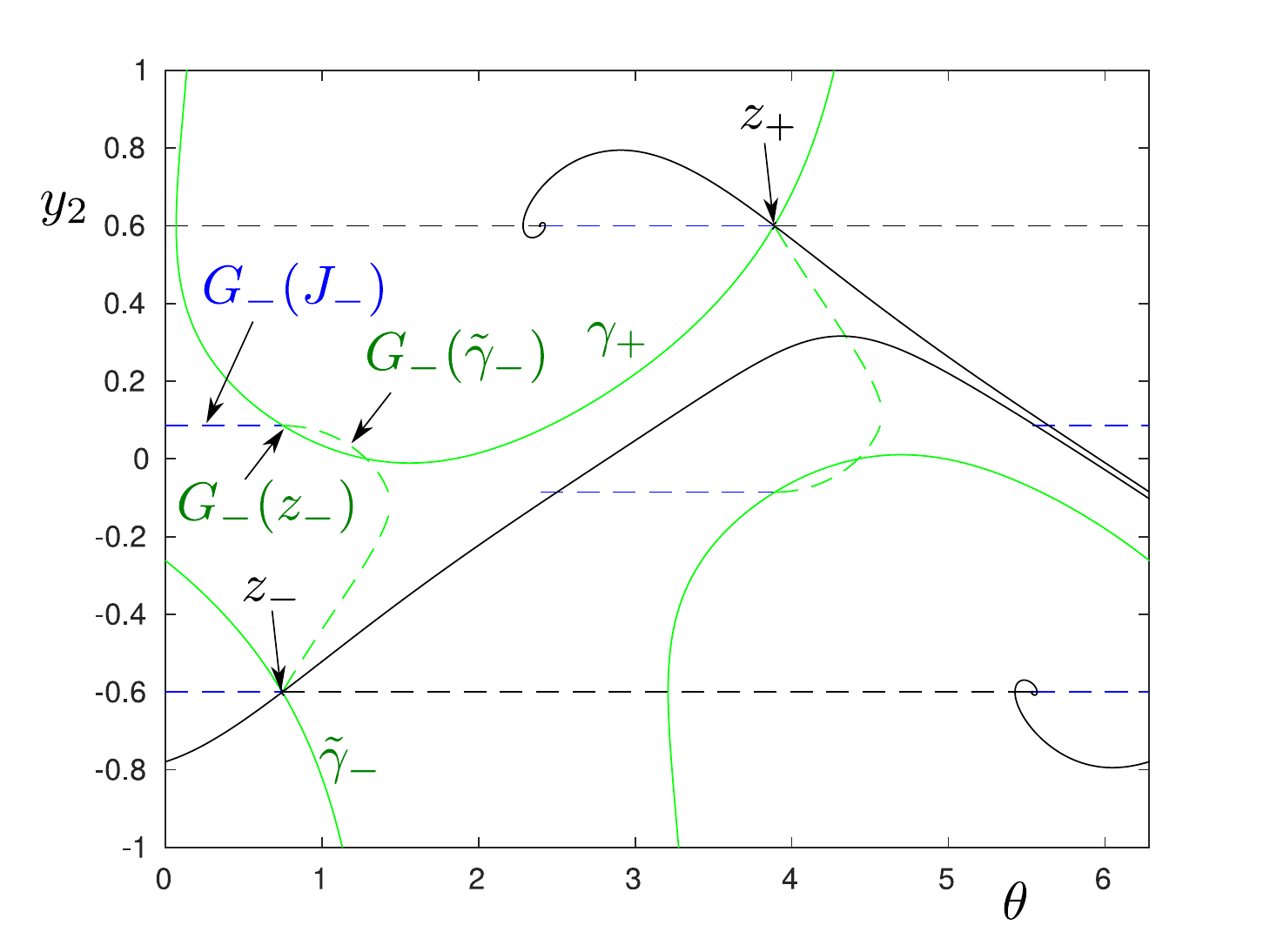}\end{subfigure}
 \begin{subfigure}{.49\textwidth}\caption{ }\label{fig:xitpdph}
  \centering
 \includegraphics[scale=0.415]{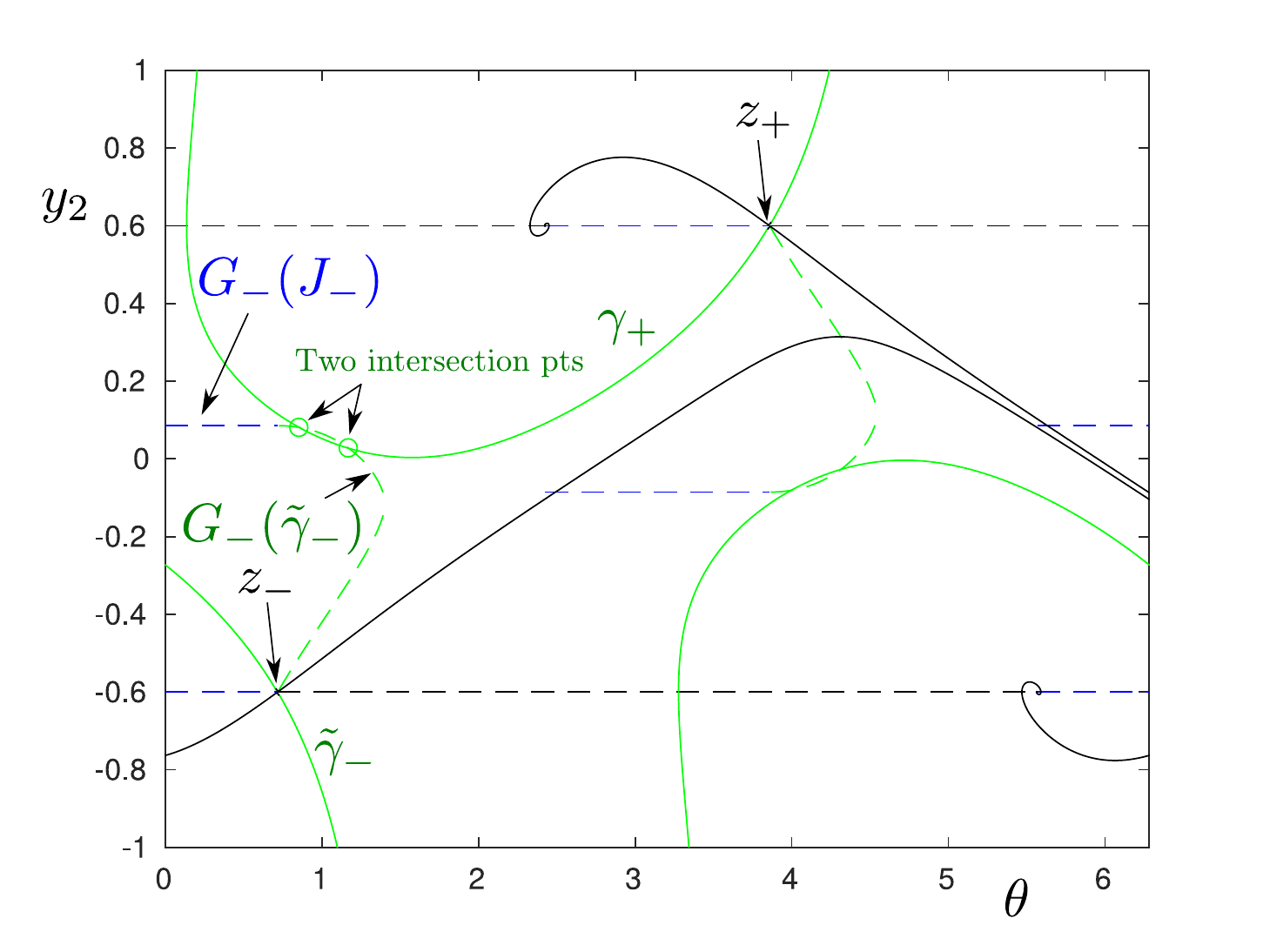}\end{subfigure}
 \begin{subfigure}{.49\textwidth}\caption{ }\label{fig:xitph}
  \centering
 \includegraphics[scale=0.415]{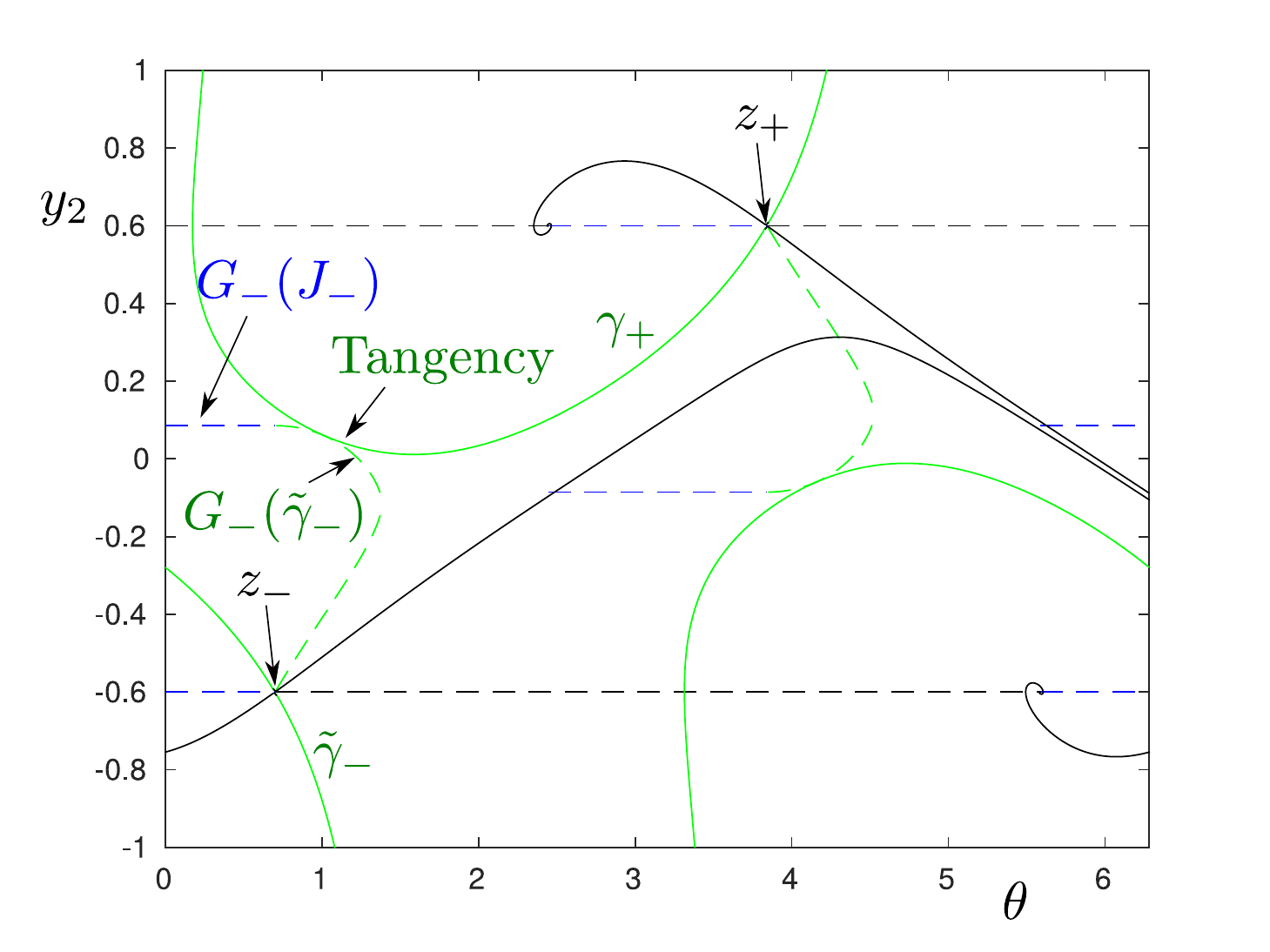}\end{subfigure}
  \caption{In (a) we show $\theta_\Upsilon$ for $\delta=0.6$, $\mu_s=1.1$ and $10$ different values $\mu_d$; the red arrow indicates the increasing direction of $\mu_d$. The values of $\mu_d$ are equidistributed within $\mu_d\in [0.4,0.925]$, with $\mu_d=0.4$ corresponding the values used in \cite{bossolini2017b}. The circles indicate the points with $\xi=\xi_{pd}$ whereas the squares indicate $\xi=\xi_t$.  These points trace out a curve in purple and cyan, respectively. (b), (c) and (d) are for $\mu_d=0.4$ and show the corresponding (computed) phase portraits for $\xi=\xi_{pd}\approx 0.8179 $, $\xi=0.7950$ and $\xi=\xi_t\approx 0.7835$, respectively.}\label{fig:Ups}
\end{figure}
\subsection{Perturbed dynamics}

By Fenichel's theory \cite{fenichel1974a,fenichel1979a,jones1995a,kuehn2015a}, any compact submanifold $S_a$ of $C_a$ perturbs to an attracting slow manifold $S_{a,\vare}$ (which we may take to be symmetric \cite{haragus2011a} such that $\mathbb S S_{a,\vare}=S_{a,\vare}$) having a stable manifold which is foliated by perturbed fibers, each smoothly $\mathcal O(\vare)$-close to the unperturbed ones of $C_a$. A similar result holds for $C_r^+$ and $C_r^-$, producing repelling slow manifolds $S_{r,\vare}^\pm$ for all $0<\vare\ll 1$. 
Moreover, the reduced flows on these manifolds are regular perturbations of the corresponding reduced flow on the critical manifolds. Finally, points on $S_{a,\vare}$ that reach a regular jump point in forward time, follows, see \cite{szmolyan2001a} and the previous analysis, the singular flow in such way that these points return to $S_{a,\vare}$ through base-points obtained as a small perturbation of the images under $G_-$. On the other hand, see \cite{szmolyan2001a}, the singular vrai canards $\gamma_-$, $\gamma_+=\mathbb S\gamma_-$ of the folded saddles perturb to perturbed canards $\gamma_{-,\vare}$ and $\gamma_{+,\vare}=\mathbb S\gamma_{-\vare}$, respectively, connecting $S_{a,\vare}$ with $S_{r,\vare}^\pm$. These orbits have an unstable foliation of fibers on the side of $S_{r,\vare}^\pm$, which following the perturbations of $G_\pm$ and $L_\pm$ produce stable foliations of base points on $S_{a,\vare}$. 

% In conclusion, the singular structure provided by: the reduced problem on $C_a$, the return mechanisms through $G_\pm$ of the regular jump points $J_\pm$ as well as the vrai canards and their return mechanisms through $G_\pm$ and $L_\pm$, gives an accurate structure for which the complete global dynamics can be characterized. 

% \section{Description of global dynamics based upon the singular structure}

\subsection{A stroboscopic mapping}

Since each point reaches a neighborhood of $C$, we consider a section $\Pi$ at $\theta=\theta_*<\cos^{-1}(\delta \xi^{-1})$ in the scaling chart, using the $(x,y_2,\theta)$-coordinates, defined in a neighborhood of $S_a\cap \{\theta=\theta_*\}$ where $S_a\subset C_a$ is a compact submanifold. We then define the associated stroboscopic mapping $P_\vare:\,\{\theta=\theta_*\}\rightarrow \{\theta=2\pi+\theta_*\}$ for all $0<\vare\ll 1$ on this neighborhood (adjusting the domains appropriately if necessary) and write this mapping in terms of $(x,y_2)$. By the symmetry of the system, we have $P_\vare = R_\vare^2$ for $R_\vare=\mathbb S\circ Q_\vare$ where $Q_\vare$ is the ``half-map'' obtained from $\{\theta=\theta_*\}$ to $\{\theta=\pi+\theta_*\}$. 

We now describe the singular map $R_0$. Let $(x,y_2,\theta_*)\in \Pi$ be a point in the aforementioned neighborhood of $S_a$. Then project $(x,y_2,\theta_*)\mapsto (x,y_{2b},\theta_*)$ onto $S_a$ using the smooth fiber projections. Next, flow $(x,y_{2b},\theta_*)$ forward using the reduced flow until either $\theta=\pi+\theta_*$ (in which case $R_0(x,y_2)$ is obtained upon applying the symmetry $\mathbb S$ to the resulting end-point) or until we reach $J_-$. In the latter case, we apply $G_-$ to obtain a new point on $S_a$, which we again flow forward until $\theta=\pi+\theta_*$ or until we reach either $J_-$ or $J_+$. In the latter case, we apply either $G_-$ or $G_+$, respectively. This process concludes after finitely many steps (it is clear that there can be no accumulation points) and the image $R_0(x,y_{2})$ is then obtained upon applying the symmetry $\mathbb S$ to the resulting end-point. It is also clear that $R_0$ is piecewise smooth. The set of discontinuities of $R_0$ is closed and includes the set of points that reach $\gamma_-$ or $\gamma_+=\mathbb S\gamma_-$ under the process describing $R_0$. 
\begin{lemma} \label{lem:R0eps}
Consider any open set $\widetilde \Pi\subset \Pi$ not including the points of discontinuity of $R_0$. Then $R_\vare\vert_{\widetilde \Pi}\rightarrow R_0\vert_{\widetilde \Pi}$ in $C^l$ for any $l\in \mathbb N$ as $\vare\rightarrow 0$. 
%  takes the following form
%  \begin{align*}
%   S\circ Q_\vare (x,y_2)\mapsto \begin{pmatrix}
%                             G(x,y_2,\vare)\\
%                             H(x,y_2,\vare)\\                            
%                            \end{pmatrix},
%  \end{align*}
% where $G(x,y_2,0)$ and $H(x,y_2,0)$ are such that $(G,H,0)\in C$
\end{lemma}
\begin{proof}
 Follows from the analysis above, Fenichel's theory and \cite{szmolyan2001a,szmolyan2004a}.
\end{proof}

% We will first describe $R_0$ in the following. Subsequently, we will focus on $R_\vare$ near the discontinuity points of $R_0$. 

% \fbox{Need to fix this: define section $\{\theta=\theta_*\}$ so that $G_-(z_-)$ does not reach $\{\theta=\theta_*+\pi\}$}
\section{Existence of stick-slip and slip periodic orbits}\label{sec:pos}

In this section, we use our geometric approach, based upon GSPT and specifically blowup, to prove existence of various families of symmetric limit cycles of \cref{eq:model0} -- including stick-slip $\Gamma_{ss}$, canard $\Gamma_c$ and pure-stick $\Gamma_{ps}$ limit cycles, see illustration in the blown down $(x,y,\theta)$-space in \Cref{fig:lcs} -- as fixed-points of $R_\vare$. 
\begin{figure}[!t]
        \centering
        \includegraphics[width=0.65\textwidth]{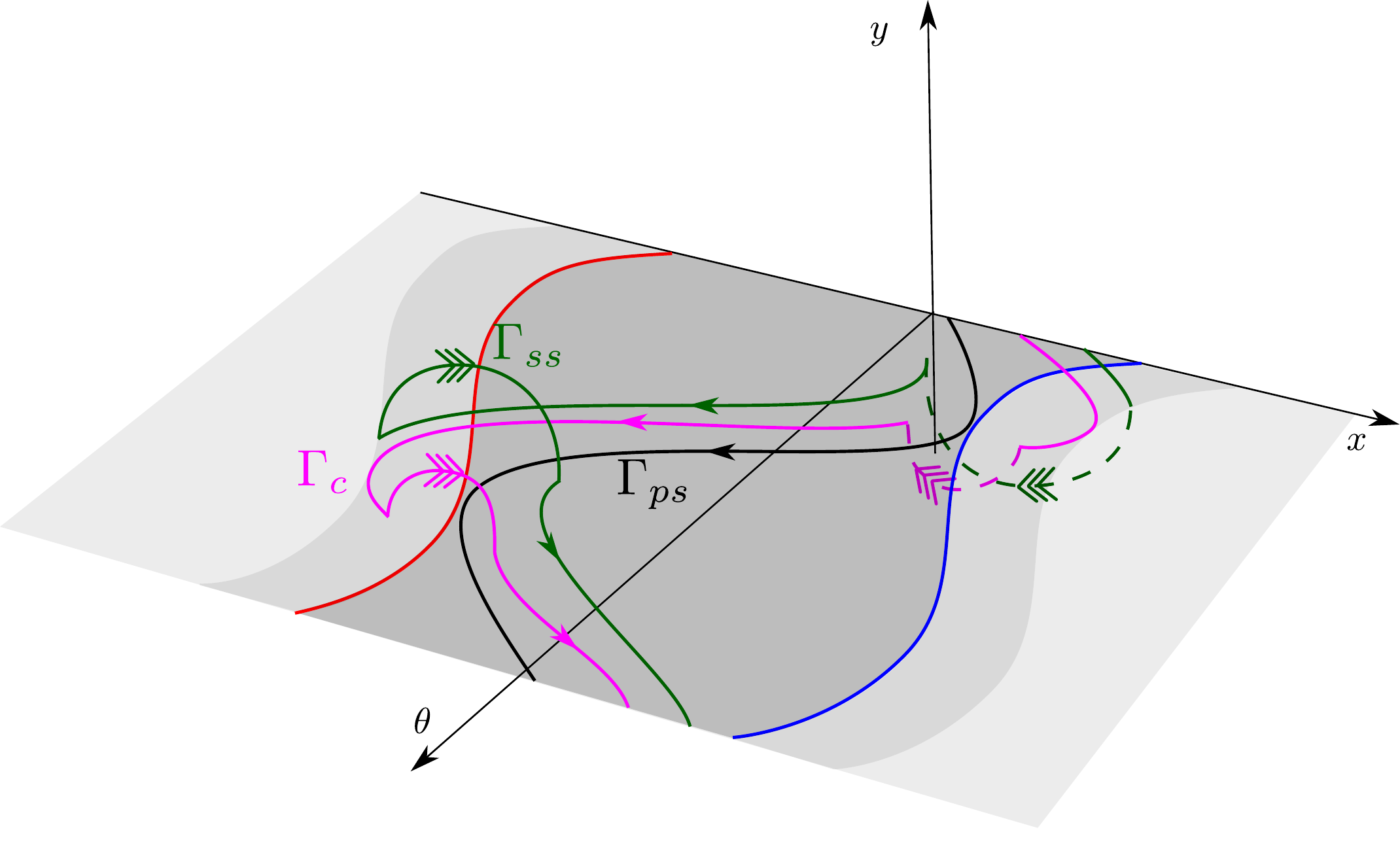}
        \caption{Examples of limit cycles in the blown down $(x,y,\theta)$-space (see also \Cref{fig:pws}) in the singular limit $\vare= 0$. The stick-slip orbit $\Gamma_{ss}$ is in green, the canard orbit $\Gamma_c$ is purple and finally the pure-stick orbit $\Gamma_{ps}$ is in black, see the rigorous statements on the existence of these limit cycles in \Cref{thm:main1new}. Notice that $\Gamma_{ss}$ and $\Gamma_{c}$ are (generalized) relaxation oscillations with slow pieces (the stick phase) on $y=0$ interspersed with fast jumps (the slip phase) within $y\gtrless 0$. For $\Gamma_{ss}$ the transition to slip from stick occurs at points corresponding to a regular jump points upon blowup (i.e. points on $J_\pm$), whereas for $\Gamma_{c}$ the onset of slip is delayed \cite{bossolini2020}. Notice that $\Gamma_c$ can also be pure-stick if at the singular level it follows $L_\pm$ rather than $G_\pm$. $\Gamma_{ps}$ in contrast is purely ``slow''; for $0<\vare\ll 1$ it is contained within an attracting slow manifold $S_{a,\vare}$.  }
        \label{fig:lcs}
\end{figure}
\begin{theorem}\label{thm:main1new}
Consider \cref{eq:model0} and suppose \ref{it:A1}-\ref{it:A5}. 
 Fix any compact intervals 
  $I_{ss}\subset (\xi_{pd},\xi_0]$, $I_{ssc}\subset (\xi_t,\xi_{pd})$, $I_{c}\subset (\xi_{t},\xi_0]$ and $I_{ps}\subset (0,\xi_0]$, $\xi_i$, $i=t,pd,0$ being described in \ref{it:A5}.  
 \begin{enumerate}
  \item Then for all $0<\vare\ll 1$ there exist four continuous families of hyperbolic symmetric limit cycles:
 \begin{align*}
  I_{ss}\ni \xi &\mapsto \Gamma_{ss}(\xi,\vare),\\
  I_{ssc} \ni \xi &\mapsto \Gamma_{ssc}(\xi,\vare),\\
  I_c \ni \xi &\mapsto \Gamma_{c}(\xi,\vare),\\
  I_{ps}\ni \xi &\mapsto \Gamma_{ps}(\xi,\vare).
 \end{align*}
 (Here subscripts indicate the following: $ss$=stick-slip, $ssc$=stick-slip with canards, $c$=canard limit cycles which can be either stick-slip or pure-stick depending on whether it follows $G_\pm$ or $L_\pm$, respectively, $ps$=pure-stick). 
\item For each $\xi \in I_{ss},I_{ssc}, I_c$, $\lim_{\vare\rightarrow 0} \Gamma_{ss}(\xi,\vare),\lim_{\vare\rightarrow 0} \Gamma_{ssc}(\xi,\vare),\lim_{\vare\rightarrow 0} \Gamma_{c}(\xi,\vare),\lim_{\vare\rightarrow 0} \Gamma_{pc}(\xi,\vare)$ are PWS cycles (see \Cref{fig:lcs}); in particular, $\lim_{\vare\rightarrow 0} \Gamma_{c}(\xi,\vare)$ as well as $\lim_{\vare\rightarrow 0} \Gamma_{ssc}(\xi,\vare)$ contain canard segments for all $\xi\in I_{c}, I_{ssc}$, respectively, whereas $\lim_{\vare\rightarrow 0} \Gamma_{ss}(\xi,\vare)$ is of stick-slip type intersecting the set of jump points $J_\pm$ each once within one period. Finally, $\lim_{\vare\rightarrow 0} \Gamma_{ps}(\xi,\vare)$ is pure-stick for all $\xi\in I_{ps}$. 
% There is always 
% \item 
\item $\Gamma_{ps}(\xi,\vare)$ and $\Gamma_{ss}(\xi,\vare)$ are hyperbolically attracting for any $\xi\in I_{ps}, I_{ss}$, respectively.
\item Finally, fix any $\xi\in(0,\xi_t)$ and any large compact set $B$ in the phase space. Then for all $0<\vare\ll 1$ there are no stick-slip orbits and $\Gamma_{ps}(\xi,\vare)$ attracts each point in $B$. 
\end{enumerate}
\end{theorem}
\begin{proof}
 The existence of the family $\Gamma_c(\xi,\vare)$ was proven in \cite[Proposition 6.6]{bossolini2020}. Indeed, suppose first that $\xi>\xi_{pd}$ such that $\gamma_+$ by assumption \ref{it:A5} intersects (the closure of) $L_-(\tilde \gamma_-)\cup G_-(\tilde \gamma_-)$ once. Then there is a singular cycle of the blown up system with desirable hyperbolicity properties, consisting of a segment of $\gamma_-$, a fast jump on the repelling side at a point determined by the aforementioned intersection point, and then a symmetric segment of $\gamma_+$ and symmetric jump. Using the geometric construction based upon GSPT in \cite[Proposition 6.6]{bossolini2020} we then obtain the canard limit cycle $\Gamma_c(\xi,\vare)$ for any $0<\vare\ll 1$. We can continue such a limit cycle continuously in $\xi$ for any $\xi>\xi_t$ since by assumption \ref{it:A5} the transverse intersection of $\gamma_+$ and $L_-(\tilde \gamma_-)\cup G_-(\tilde \gamma_-)$ persists up until $\xi=\xi_t$. $\Gamma_{ssc}(\xi,\vare)$ is handled similarly from the intersection of  $\gamma_+$ with $G_-(\tilde \gamma_-)$ appearing after $\xi=\xi_{pd}$.

 We therefore proceed to prove the existence of the families $\Gamma_{ss}(\xi,\vare)$ and $\Gamma_{ps}(\xi,\vare)$ and the properties described in the theorem. For this we use \Cref{lem:R0eps} and proceed to analyze the mapping $R_0$.
It is without loss of generality to restrict $R_0$ to $C_a$ and in this way $R_0$ just becomes a mapping of $y_2$. We shall adopt this convention henceforth and therefore write $R_0$ as a mapping on $y_2$: $y_2\mapsto y_{2+}=R_0(y_2)$. 
It is obvious that for a $y_2\in \tilde \Pi$ (i.e. the set where $R_0$ is smooth), then $R_0$ is decreasing and it is contracting there, i.e. $R_0'(y_2)\in (-1,0)$. To see the latter, notice that $R_0$ is obtained from motion along $C_a$, jumps at either $J_-$ or $J_+$ and a reflection $y_2\mapsto -y_2$ due to $\mathbb S$. The maps $G_\pm$ act like ``translations'' and does not affect the stability. In particular, we may locally identify points on $J_\pm$ with those on $G_\pm (J_\pm)$ through the mappings $G_\pm$. Then the reduced dynamics becomes continuous at $J_\pm$ and piecewise smooth. The contracting properties of $R_0$ then follows from that the divergence of \cref{eq:reddes} is $-1$. 

We first focus attention on the existence of $\Gamma_{ss}(\xi,\vare)$ for $\xi>\xi_{pd}$. For this we define the section $\Pi$ by setting 
\begin{align}
\theta_*=\theta_{\Upsilon}(\xi^{-1}).\label{eq:thetaStar}
\end{align}
Let $y_{2,c}$ denote the $y_2$-value of the first intersection of $\gamma_-\cap \Pi$ obtained by flowing the local stable manifold backwards on $C_a$. It is a point of discontinuity of $R_0$, where by assumption  
\begin{align*}
 \lim_{y_2\rightarrow y_{2,c}^-} R_0(y_{2})<y_{2,c},\quad \lim_{y_2\rightarrow y_{2,c}^+} R_0(y_{2})>y_{2,c}.
\end{align*}
By \cref{eq:thetaStar}, we have $\lim_{y_2\rightarrow -\delta^+}R_0(y_2)=y_{2,c}$, see also \Cref{fig:reduceda}. Let $K=(-\delta,y_{2,c})$. Then since $R_0$ contracts it follows that at the singular level, the forward flow of any $(y_2,\theta_*)$ with $y_2\in K$ first follows the slow flow on $C_a$, jumps at $J_-$ to a point on $G_-(J_-)$, and then finally follows the slow on $C_a$ again up until $\theta=\theta_*+\pi$. Therefore $R_0(K)\subset (-\delta,y_{2,c})$ and 
\begin{align*}
-\delta< \lim_{y_2\rightarrow y_{2,c}^-} R_0(y_2)<y_{2,c},
%  -\delta\ge \lim_{y_2\rightarrow a^+} R_0(y_2) <\lim_{y_2\rightarrow y_{2,c}^-} R_0(y_2) <y_{2,c}.
\end{align*}
There is therefore a unique fixed point of $R_0$ inside $K$ by the contraction mapping theorem. We may also see this more directly by the intermediate value theorem: The continuous function $V$ on $K$ defined by $V(y_2):=R_0(y_2)-y_2$ satisfies
\begin{align*}
 \lim_{y_2 \rightarrow -\delta^+} V(y_2)>0>\lim_{y_2 \rightarrow y_{2,c}^-} V(y_2).
\end{align*}
Using \Cref{lem:R0eps} we perturb the fixed-point of $R_0$ into an attracting fixed-point of $R_\vare$ for all $0<\vare\ll 1$. It corresponds to a stick-slip periodic orbit $\Gamma_{ss}(\xi,\vare)$. Seeing that the limit cycle is obtained from an implicit function theorem argument, $\Gamma_{ss}$ depends smoothly on $\xi\in I_{ss}$. 

% To see that $\Gamma_{ss}(\xi,\vare)$ is attracting for each $\xi\in I_{ss}$ and all $0<\vare\ll 1$, we first notice that at the singular level it consists of motion along $C_a$, a jump along $J_-$ described by $G_-(J_-)$ followed by motion along $C_a$. The map $G_-(J_-)$ acts like ``translation'' and does not affect the stability. In particular, we may identify points on $J_-$ with those on $G_-(J_-)$ through the mapping $G_-$. Then the reduced dynamics becomes continuous at $J_-$ and piecewise smooth. We know that a limit cycle of a $C^1$ system $dy/d\theta=f(y,\theta)$, with $f$ periodic in the second argument, is attracting if $f_y'(y,\theta)<0$. But this result generalises easily to our continuous piecewise smooth case and consequently, since the divergence of \cref{eq:reddes} is $-1$, we conclude that the fixed point is attracting for $R_0$. It then follows from \Cref{lem:R0eps} that the perturbed fixed point is attracting for $R_\vare$ also whenever $0<\vare\ll 1$.   

The family $\Gamma_{ps}(\xi,\vare)$ is obtained by the same approach, except that we now fix the section $\Pi$ using $\theta_*=\theta_{-}(\xi^{-1})$. Let $K=(-\delta,b)$ with $b$ defined as follows: If the backward flow of the local stable manifold $\gamma_+$ intersects $\Pi$ before it intersects $\{y_2=\delta\}$ then $b$ is defined as the $y_2$-coordinate of the intersection $\Pi\cap \gamma_+$. Otherwise, $b=\delta$. In any case, by construction of the section $\Pi$ and by the symmetry $\mathbb S$ no points in $K$ reach $J_\pm$ and we have that $R_0(K)\subset (y_{2,c},b)$, that $R_0$ is monotonically decreasing and that it is contracting on $K$. We therefore obtain a unique and attracting fixed-point by the contraction mapping theorem, which we by \Cref{lem:R0eps} perturb into an attracting limit cycle $\Gamma_{ps}(\xi,\vare)$ for all $0<\vare\ll 1$ by the implicit function theorem. $\Gamma_{ps}(\xi,\vare)$ is pure-stick since it belongs to a slow manifold $S_{a,\vare}$. Clearly, this argument also holds for $0<\xi\le \delta$ in which case $J_\pm$ disappear and $C_a$ is forward invariant and hence the $\omega$-limit set is $\Gamma_{ps}$ for any fixed $\xi \in (0,\xi_t)$ and all $0<\vare\ll 1$. 
% But then since $R_0$ is monotonically decreasing, 
% the function $V(y_2):=R_0(y_2)-y_2$, satisfying
% \begin{align*}
%  s
% \end{align*}
\end{proof}

\begin{figure}[!t]
        \centering
        \includegraphics[width=0.75\textwidth]{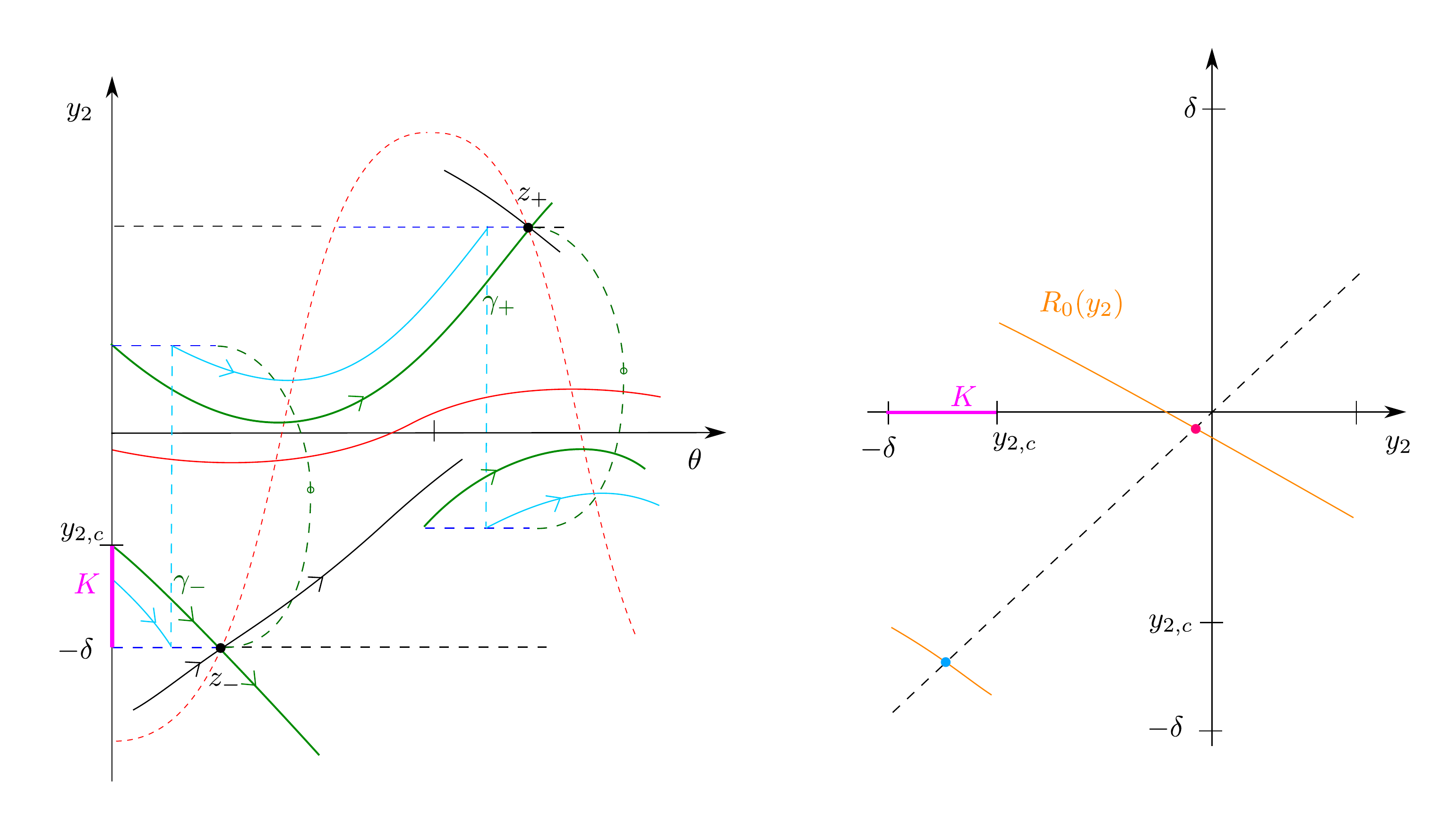}
        \caption{Construction of limit cycles through the mapping $R_0$. }
        \label{fig:reduceda}
\end{figure}

We expect that the families $\Gamma_{ss}(\xi,\vare)$, $\Gamma_{ssc}(\xi,\vare)$ and $\Gamma_{c}(\xi,\vare)$ are connected in a one-parameter family of limit cycles, that contains each of the cycles $\Gamma_{ss}, \Gamma_{ssc}$ and $\Gamma_c$ within their domain of existence provided by \Cref{thm:main1new}, having a fold bifurcation at $\xi=\xi_t+o(1)$ and a period-doubling bifurcation for $\xi=\xi_{pd}+o(1)$. See \Cref{fig:bifGamma} for an illustration. But we have not pursued a rigorous statement of this kind. %In this way, we speculate that there is one-parameter family of limit cycles that contain each each of the cycles $\Gamma_{ss}, \Gamma_{ssc}$ and $\Gamma_c$ within their domain of existence provided by \Cref{thm:main1new}. 
However, since there can be no stick-slip limit cycles for $\xi<\xi_t$ it follows indirectly that there has to be some sort of fold bifurcation but it is not obvious whether the branches really are connected and whether there could be additional folds. Certainly, the fold bifurcation cannot be a saddle-node bifurcation in any meaningful sense of the word due to the chaotic dynamics we prove in \Cref{thm:main2} in the following section. 
Nevertheless, we can relatively easy explain why we expect to find a period doubling bifurcation:
When $\gamma_+$ intersects $G_-(J_-)$ for $\xi>\xi_{pd}$, the stick-slip limit cycle $\Gamma_{ss}(\xi,\vare)$ is attracting cf. \Cref{thm:main1new}, but just on the other side of $\xi<\xi_{pd}$, the stick-slip limit $\Gamma_{ssc}(\xi,\vare)$ is of saddle-type. In fact, in the latter case, we almost immediately see from the singular structure that $\mathcal O(e^{-c/\vare})$-displacements from the corresponding fixed point of $R_\vare$ along the negative $y_2$-direction tangent to $S_{a,\vare}$, leave a vicinity of $C_r^-$ before the corresponding jump point of the limit cycle. Upon following $G_-$ these displacements then evolve ``below'' the corresponding limit cycle and therefore produce $\mathcal O(1)$-displacements at $\theta=\pi+\theta_*$ in the negative $y_2$-direction. Upon applying the symmetry $\mathbb S$, we see that these displacements then lead to very negative eigenvalues of the linearization of $R_\vare$. Although, we have not pursued this in details, the transition from attracting to very repelling is expected to be monotone. In combination this then explains the observed existence of a (locally unique) period doubling bifurcation near the the $\xi$-value $\xi_{pd}$ where $\gamma_+$ intersects the $G_-$ image of the folded saddle $z_-$. We demonstrate that these considerations are in agreement with our numerical computations in the following section.% \Cref{fig:coexistence_bifnew} for a smaller value of $\varepsilon$. See figure caption for further explanation.

\begin{figure}[!t]
        \centering
        \includegraphics[width=0.65\textwidth]{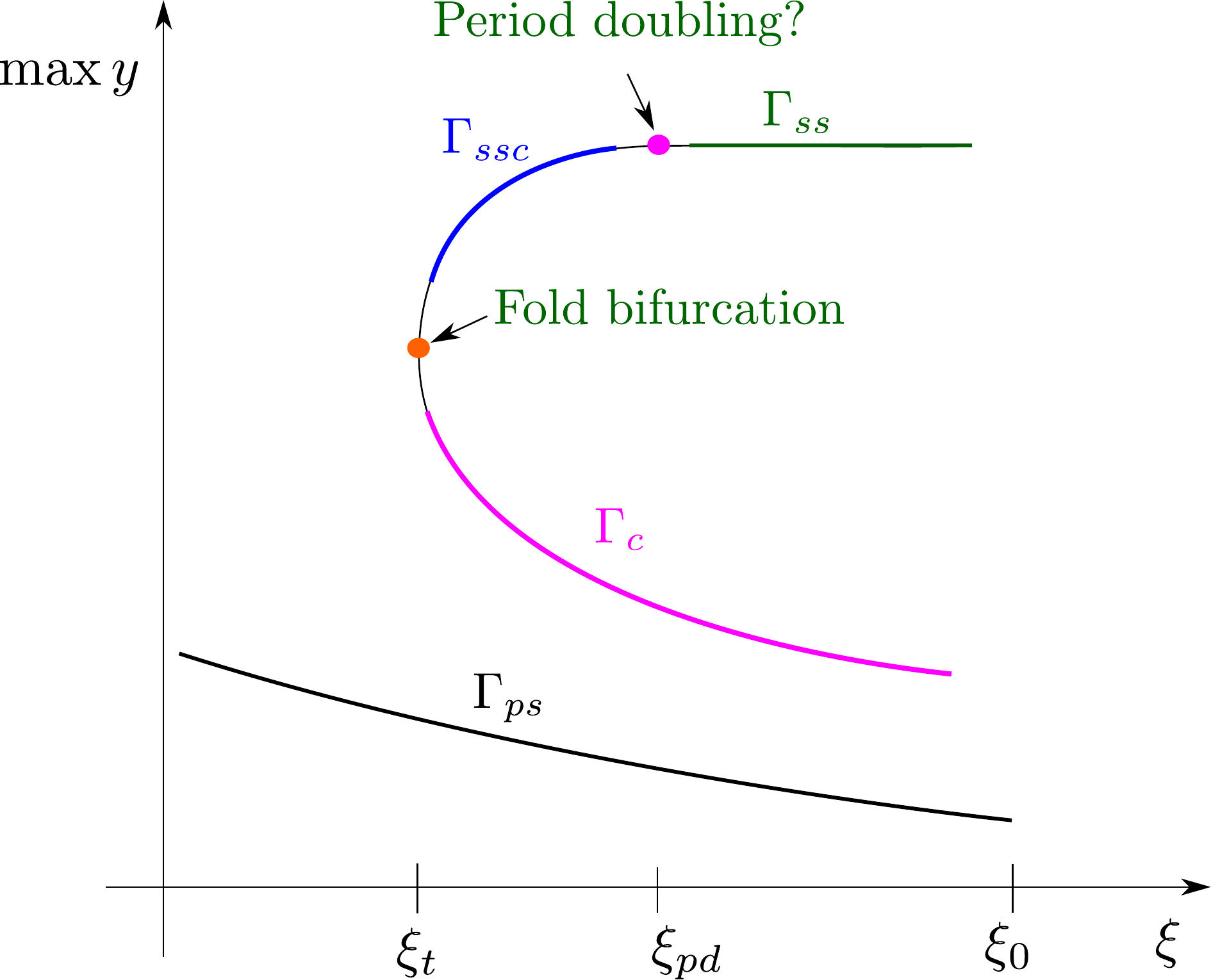}
        \caption{Illustration of the different families of limit cycles obtained by \Cref{thm:main1new} using $\xi$ as our bifurcation parameter and representing each limit cycle by $\max y$. We expect the limit cycles $\Gamma_{ss}$, $\Gamma_{ssc}$ and $\Gamma_{c}$ belong to one family of limit cycles that undergo a period doubling bifurcation at $\xi=\xi_{pd}$ and a fold bifurcation at $\xi=\xi_t$ in the limit $\vare\rightarrow 0$. }
        \label{fig:bifGamma}
\end{figure}

\subsection{Bifurcation of limit cycles}
% NEED TO REWRITE THIS.
In \Cref{fig:coexistence}, we illustrate the results of numerical computations (using the bifurcation software system AUTO \cite{Doedel2006} as well as Matlab) of \cref{eq:model0} for the regularization function \cref{eq:phie} with the parameters \cref{eq:parahere} also used in \cite{bossolini2017b,bossolini2020}.
% \begin{align*}
%  \phi(s) = \frac{x}{\sqrt{x^2+1}}\left(1+\frac{\alpha }{1+\beta x^2}\right),
% \end{align*}
% for which $k=2$, recall the assumption \ref{it:A4}, with
%  $\alpha=7.933$ and $\beta=2.2936$ such that 
% \begin{align}
%  \delta = 0.6,\,\mu_s/\mu_d = 2.75.\label{eq:parahere}
% \end{align}
%  in \Cref{fig:coexistence}.
% Notice for this regularization function
More specifically, \Cref{fig:coexistence_bif} shows a bifurcation diagram of limit cycles for $\vare=0.01$ using $\xi$ as a bifurcation parameter. \Cref{fig:coexistence_cAndss_ytheta} shows the corresponding symmetric periodic orbits in the $(\theta,y_2)$-plane in red and magenta (the corresponding points are also indicated in \Cref{fig:coexistence_bif}) for $\xi=0.9397$. In \Cref{fig:coexistence_cAndss_ytheta} we also show a pure-stick orbit in black along with the relevant ``singular objects'' described above: the singular vrai canards $\gamma_\pm$ in green, the faux canard in black, the sets $G_\pm (J_\pm)$ as well as $G_\pm (\tilde \gamma_\pm)$ and $L_\pm (\tilde \gamma_\pm)$ (green and dashed). The existence of the three limit cycles in magenta (stick-slip, i.e. belonging to the branch $\Gamma_{ss}(\xi,\vare)$), black (pure-stick, i.e. $\Gamma_{ps}(\xi,\vare)$) and red (canard type, i.e. $\Gamma_{c}(\xi,\vare)$) follow from \Cref{thm:main1new}. 
% 
% 
% Indeed, \cite[Proposition 6.6]{} studied this case and showed that this singular condition implied the existence of a periodic orbit with canards. 
The computations done in \cite{bossolini2020}, for a slightly different regularization function and a separate scaling, show that the lower branch (i.e. $\Gamma_{c}(\xi,\vare)$) consisting of canard orbits like the one in red in \Cref{fig:coexistence_cAndss_ytheta} is connected to pure-slip periodic orbits. These computations also  demonstrated a fold bifurcation of limit cycles, which we for our regularization function indicate as the cyan square in \Cref{fig:coexistence_bif}. In \Cref{fig:coexistence_csn_ytheta} we show the corresponding nonhyperbolic periodic orbit for the parameter value $\xi=0.7805$ along with the coexisting pure-stick orbit, again in black. Notice that this value of $\xi$ at the fold is close to the value of $\xi_t\approx 0.7835$ found above, recall \Cref{fig:Ups}. We also see this in \Cref{fig:coexistence_csn_ytheta} where the degenerate periodic orbit follows $\gamma_+$ and is almost tangent to $G_-(\tilde \gamma_-)$. 

% as might be expected from \cite[Proposition 6.6]{bossolini2020}, \Cref{thm:main1new} and the discussion proceeding it, that the bifurcation of limit cycles appears due to a tangency of $G_-(\tilde \gamma_-)$ with $\gamma_+$. We study the dynamics near this tangency in further details in \Cref{sec:chaos} below. 
Next, we emphasize that, as we follow the magenta limit cycle for smaller values of $\xi$ towards the fold bifurcation, a period-doubling bifurcation occurs around $\xi=0.8731$. This bifurcation produces two branches indicated in \Cref{fig:coexistence_bif}. On these branches we find different periodic orbits that are related by the symmetry. Two are illustrated in \Cref{fig:coexistence_cAndss_ytheta} in blue and cyan, the corresponding points in the bifurcation diagram are also indicated in \Cref{fig:coexistence_bif}. We see that they have (small) canard segments. We demonstrate that these considerations are in agreement with our numerical computations in \Cref{fig:coexistence_bifnew} for a smaller value of $\varepsilon$. See figure caption for further explanation. Notice in particular, that for this value of $\varepsilon=0.005$ we find that the period doubling bifurcation occurs at $\xi=0.8096$, which is relatively close to the computed value of $\xi_{pd}=0.8179$, recall \Cref{fig:Ups}. At the same time, the fold bifurcation now occurs at $\xi=0.7820$, which is also closer to the expected value to $\xi_{t}= 0.7835$. We did not manage to compute the branches from the period doubling bifurcation as in \Cref{fig:coexistence_bif} for this smaller values of $\vare>0$.

\begin{figure}[t!]
        \centering
        \begin{subfigure}{.45\textwidth}\caption{ }\label{fig:coexistence_cAndss_ytheta}
  \centering
\includegraphics[scale=0.415]{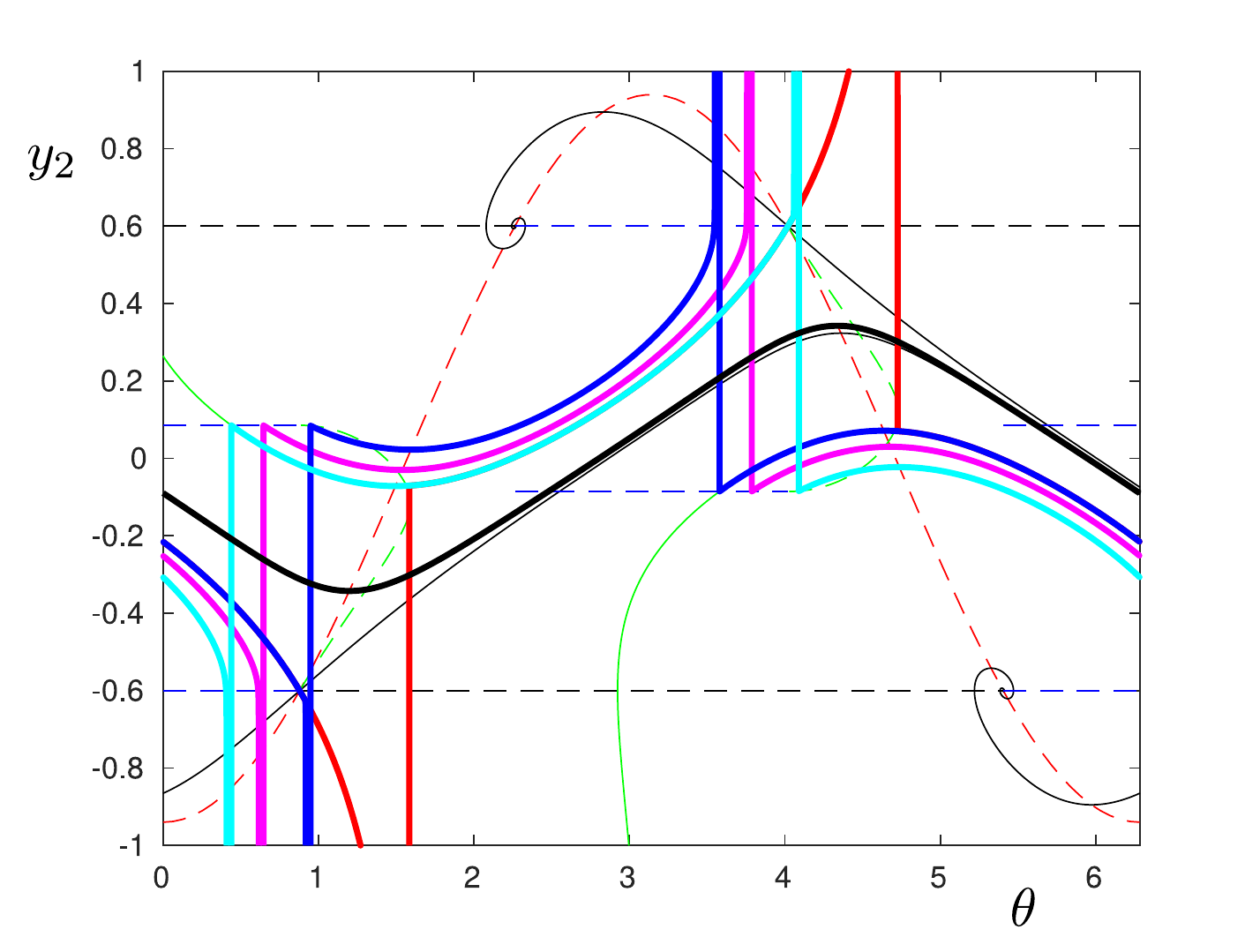}\end{subfigure}
% \begin{subfigure}{.45\textwidth}\caption{ }\label{fig:coexistence_cAndss3d}
%   \centering
%  \includegraphics[scale=0.215]{./coexistence_cAndss_ytheta.pdf}\end{subfigure}
 \begin{subfigure}{.45\textwidth}\caption{ }\label{fig:coexistence_bif}
  \centering
 \includegraphics[scale=0.415]{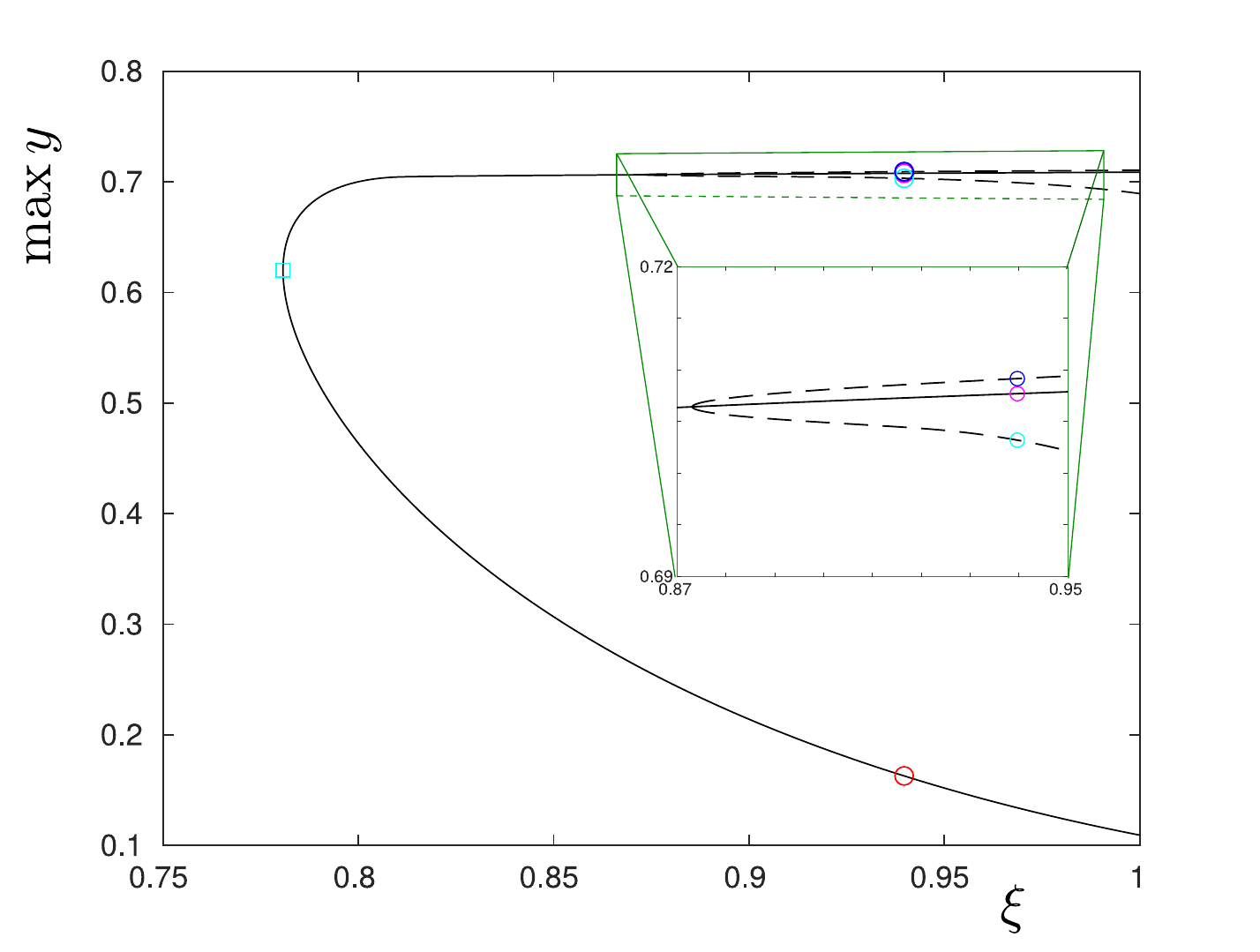}\end{subfigure}
  \begin{subfigure}{.45\textwidth}\caption{ }\label{fig:coexistence_csn_ytheta}
  \centering
 \includegraphics[scale=0.415]{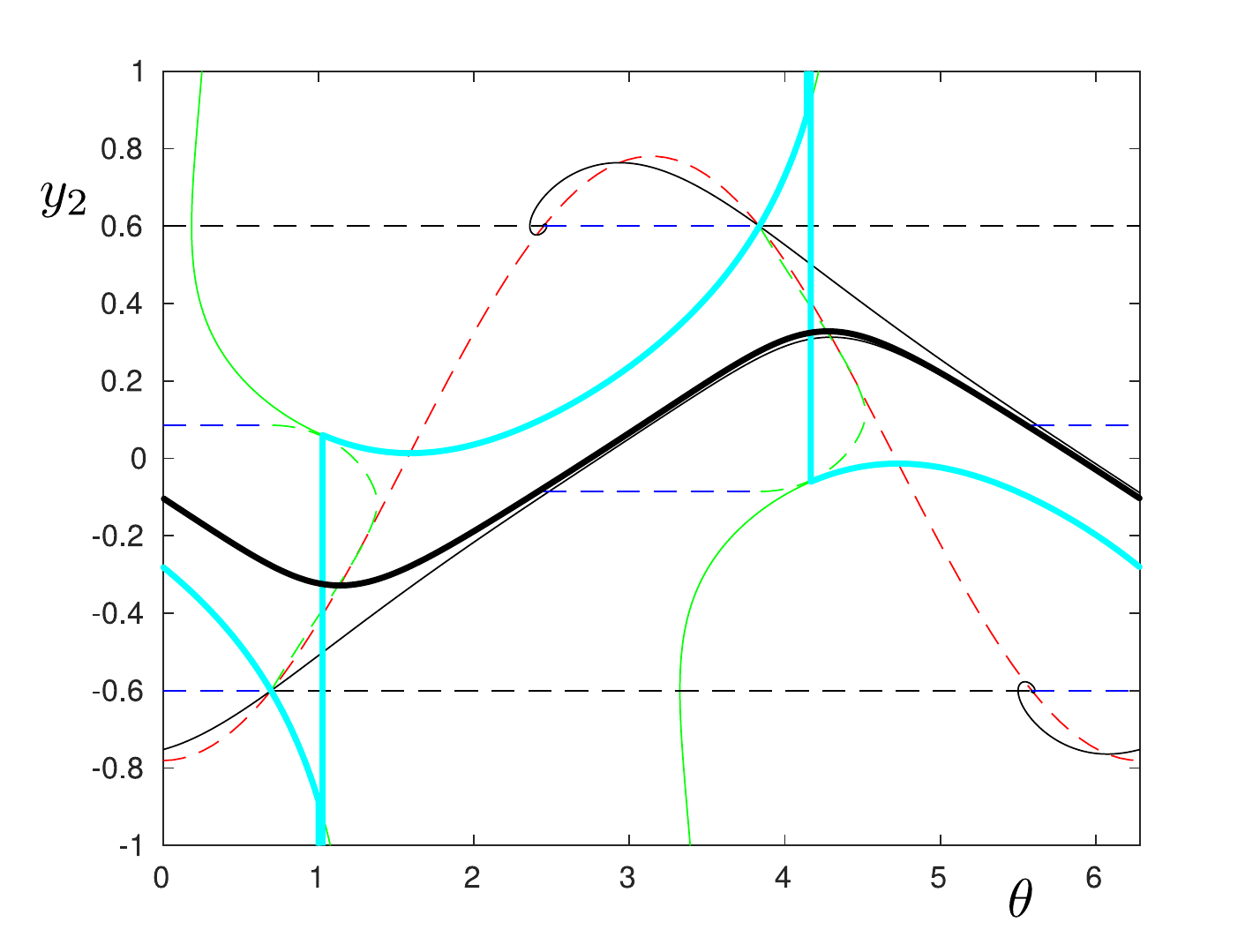}\end{subfigure}
  \begin{subfigure}{.45\textwidth}\caption{ }\label{fig:coexistencexvst}
  \centering
 \includegraphics[scale=0.415]{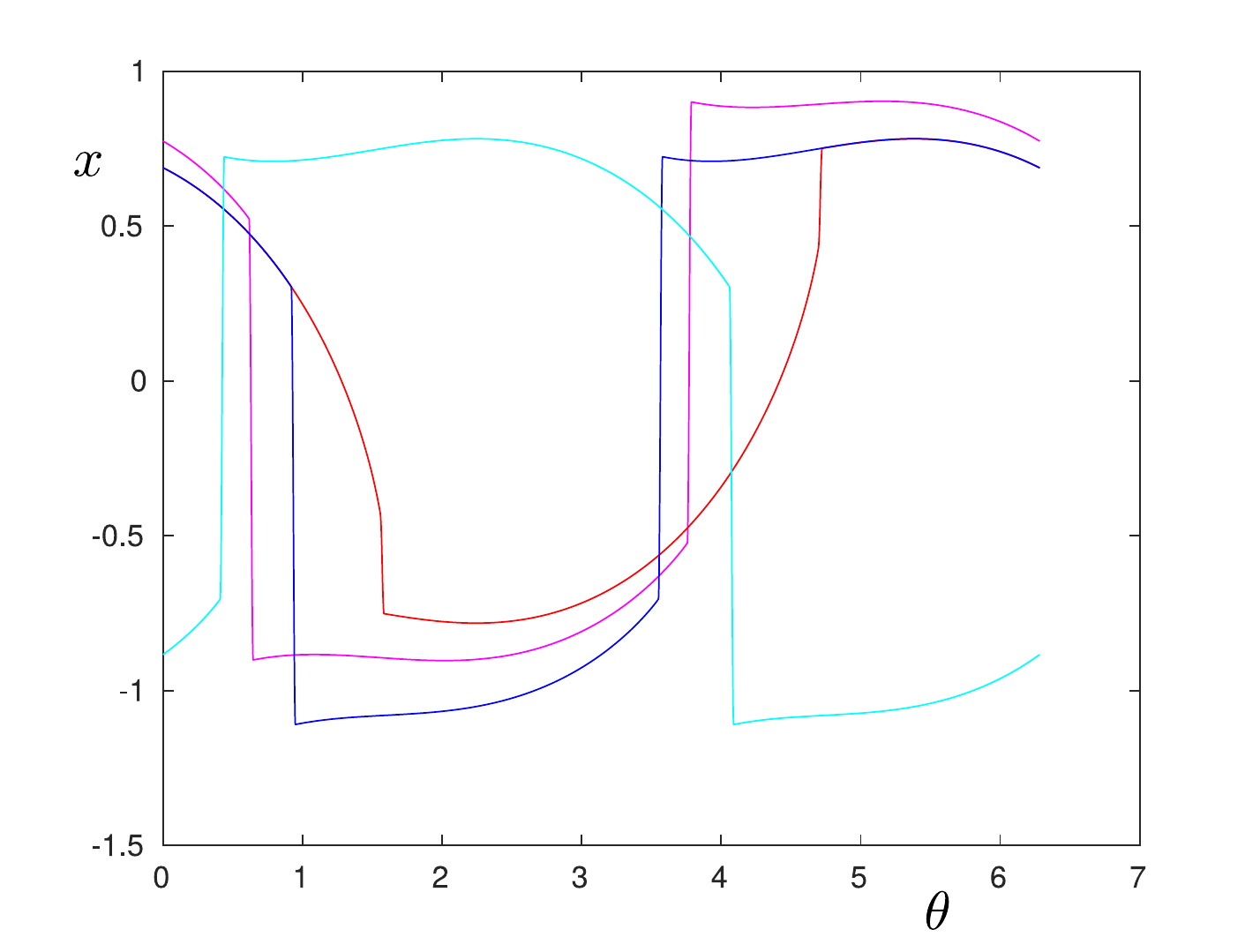}\end{subfigure}
  \caption{ In (a) we use a projection onto $(\theta,y_2)$ to illustrate five coexisting periodic orbits for $\xi=0.9397$. In magenta and black are shown stick-slip and pure-stick periodic orbits, respectively, the existence of which can be predicted by \Cref{thm:main1new}. The red orbit is a symmetric periodic orbit having a (long) canard segment, the existence of such orbits were predicted in \cite[Proposition 6.6]{bossolini2020}. Finally, in blue and cyan we illustrate nonsymmetric periodic orbits that appear as a result of a period doubling bifurcation, shown in the bifurcation in (b), see zoom. This bifurcation diagram represents each stick-slip periodic orbit as a point using the same colour for the corresponding value of $\xi=0.9397$. It also shows the fold bifurcation (cyan square) at $\xi=0.7805$ (which is close to the value of $\xi_t\approx 0.7835$, recall \Cref{fig:Ups}); (c) shows the corresponding bifurcating periodic orbit using the same colour. We see that the bifurcation occurs to good accuracy when the canard  (in green) is tangent to the set $G_-(\tilde \gamma_-)$ (dashed green line). Finally, (d) shows $x(\theta)$ for the different periodic orbits in (a). The remaining parameters are defined by \cref{eq:parahere} and $\varepsilon = 0.01$.  }\label{fig:coexistence}
\end{figure}

 \begin{figure}[t!]
        \centering
        \begin{subfigure}{.45\textwidth}\caption{ }\label{fig:coexistence_cAndss_ythetanew}
  \centering
\includegraphics[scale=0.415]{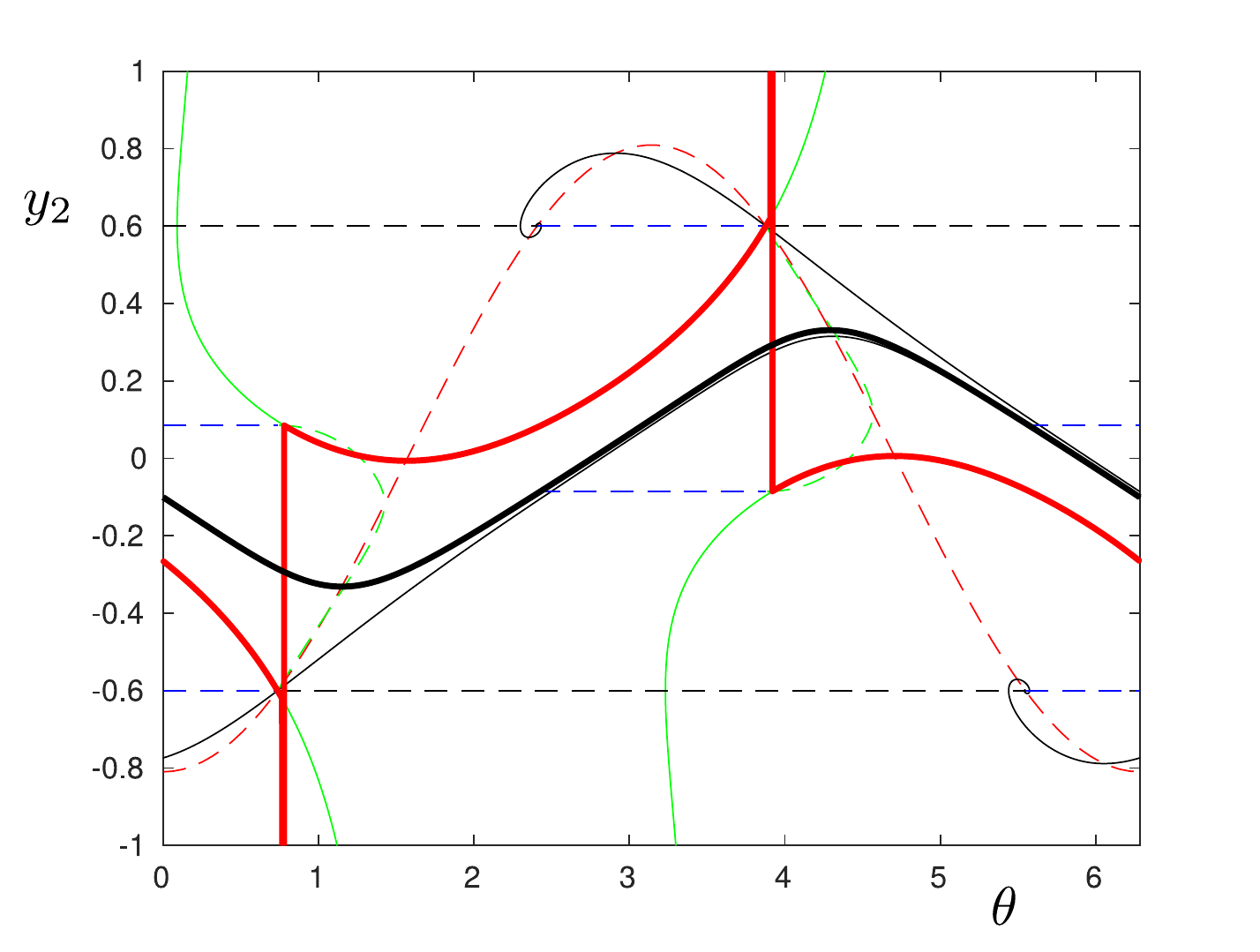}\end{subfigure}
% \begin{subfigure}{.45\textwidth}\caption{ }\label{fig:coexistence_cAndss3d}
%   \centering
%  \includegraphics[scale=0.215]{./coexistence_cAndss_ytheta.pdf}\end{subfigure}
 \begin{subfigure}{.45\textwidth}\caption{ }\label{fig:coexistence_bifnew}
  \centering
 \includegraphics[scale=0.415]{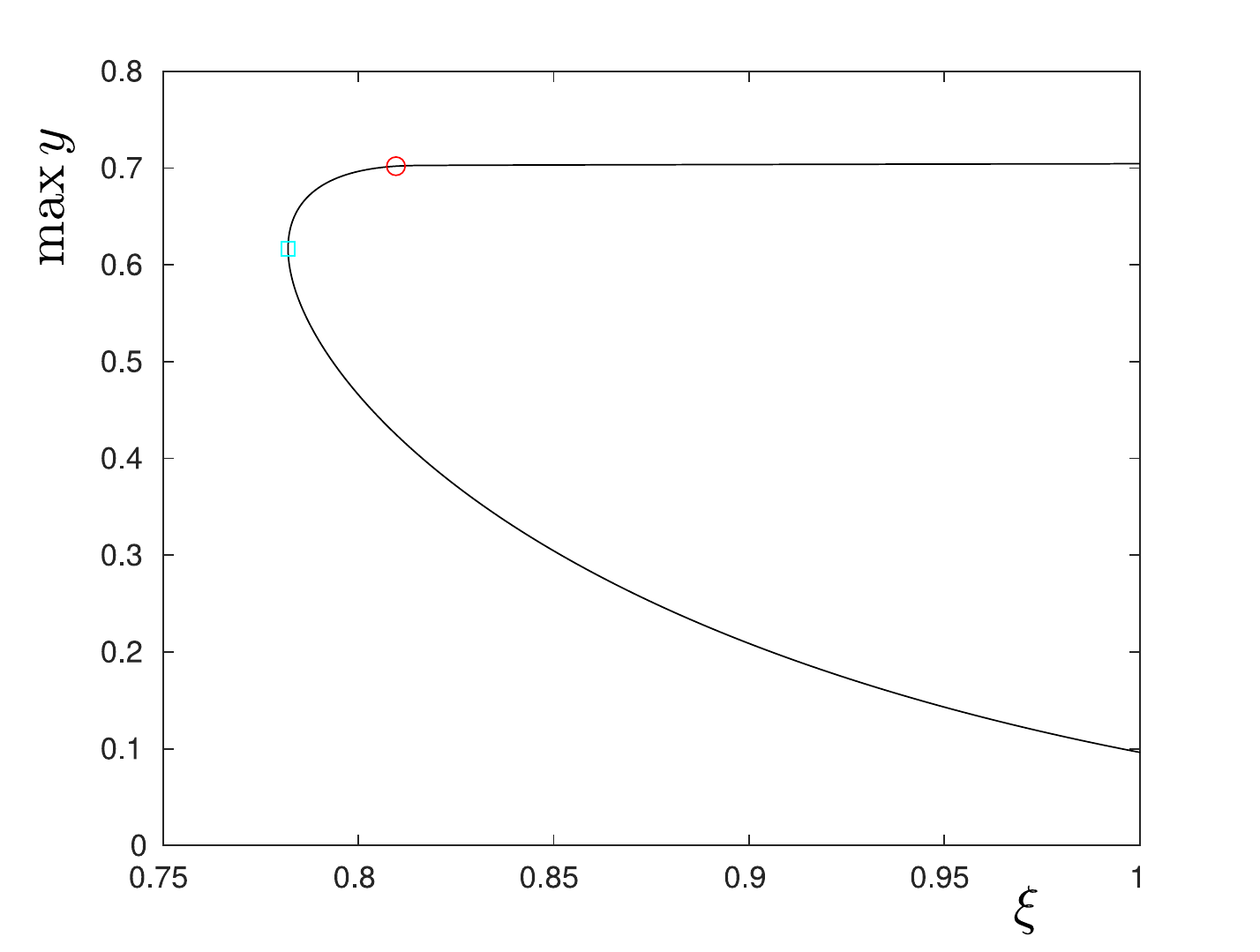}\end{subfigure}
   \caption{Same in \Cref{fig:coexistence} but for $\varepsilon=0.005$. Specifically, (b) shows the bifurcation diagram, the red point corresponding to a period doubling bifurcation of $R_\vare$ at $\xi=0.8096$. The associated degenerate periodic orbit is shown in (a) also in red along with the black stick orbit. As indicated the period doubling bifurcation occurs near $\xi=\xi_{pd}\approx 0.8179$ which is the value of $\xi$ for when the $G_-$-image of $z_-$ intersects $\gamma_+$ (where the green and blue dashed lines meet). The fold bifurcation indicated by the cyan square occurs at $\xi=0.7820$, which is also closer to the expected value to $\xi_{t}= 0.7835$.}
  \label{fig:coexistencenew}
\end{figure}
% \subsection{Saddle-node bifurcation of symmetric limit cycles}
%  In this section, we describe the saddle-node bifurcation of limit cycles in \Cref{fig:?}. 

%  \subsection{More complicated stick-slip orbits}
 
 \begin{remark}\label{rem:complicated}
When \ref{it:A5} does not hold, the dynamics can be different. \Cref{fig:complicatedytheta} shows an example for $\mu_s<1$, see the figure caption for details, in which case there is a single stick-slip periodic orbit visiting $y<-c$ and $y>c$ for $c>0$ small enough more than once within each period. In friction terms, the mass slips six times within each period, three times to the right and three times to the left. In fact, in this particular case $\xi=0.9$ it enters both regions three times each period and since $\gamma_+$ leaves $C_a$ at $F_-$ without intersecting $G_-(\tilde \gamma_-)\cup L_-(\tilde \gamma_-)$ there cannot be any canard-type limit cycles of the form described in \Cref{thm:main1new} and no pure-stick orbits either. \Cref{fig:complicatedxvst} shows the corresponding time history $x(\theta)$. However, what we see as we decrease $\xi$ slightly (not shown) from the value of $\xi$ in \Cref{fig:complicatedytheta} is that first one of each of the transitions into $y<-c$ and $y>c$ approaches the two canards $\gamma_\pm$ (in green in \Cref{fig:complicatedytheta}), respectively,  and subsequently, through the canard and the return mappings $G_\pm$ and $L_\pm$, there is an apparent smooth connection to stick-slip limit cycles with just two excursions into $y<-c$ and $y>c$ during each period. This occurs around $\xi\approx 0.83$ but we have not managed to obtain satisfactory bifurcation diagrams in AUTO so we do not show further diagrams here; we believe that the situation can be understood from the case illustrated in \Cref{fig:complicatedytheta} where the last regular jumps on $J_\pm$ land quite close to $\gamma_\pm$, respectively. The bifurcation we describe in words is reminiscent of spike-adding bifurcations in bursting models of neuroscience, see e.g. \cite{carter2020a} which has a rigorous GSPT-based treatment of this situation. In our friction setting a ``spike'' corresponds to a slip phase. As for \Cref{fig:coexistence} this canard-phenomena is accompanied by a period-doubling bifurcation. We have not pursued any of this rigorously in our setting. Finally, we note that as we decrease $\xi$ further then there is one more transition from two excursions into each region $y<-c$ and $y>c$ during each period to just one single excursion. This occurs around $\xi\approx 0.636$. \Cref{fig:complicatedytheta_xi0_61} shows one example of such a limit cycle with just two slip phases for $\xi=0.61$ slightly below this value and close to the critical $\delta=0.6$, where the folded singularities disappear and beyond which only pure-stick orbits persist. Another interesting aspect of  \Cref{fig:complicatedytheta_xi0_61} is that two folded singularities are now folded nodes and we see in \Cref{fig:complicatedytheta_xi0_61} the classical small oscillations that occur here \cite{wechselberger2005a,kristiansenthis}. 
\end{remark}
% Finally, we emphasize that there can be non-symmetric periodic orbits, being fixed points of (iterates of) $P_\vare$ rather than $R_\vare$. These orbits appear not to be accessible by our analysis, but all numerical simulations we have performed produce stable stick-slip and pure-stick periodic orbits only. 

\begin{figure}[t!]
        \centering
        \begin{subfigure}{.45\textwidth}\caption{ }\label{fig:complicatedytheta}
  \centering
\includegraphics[scale=0.415]{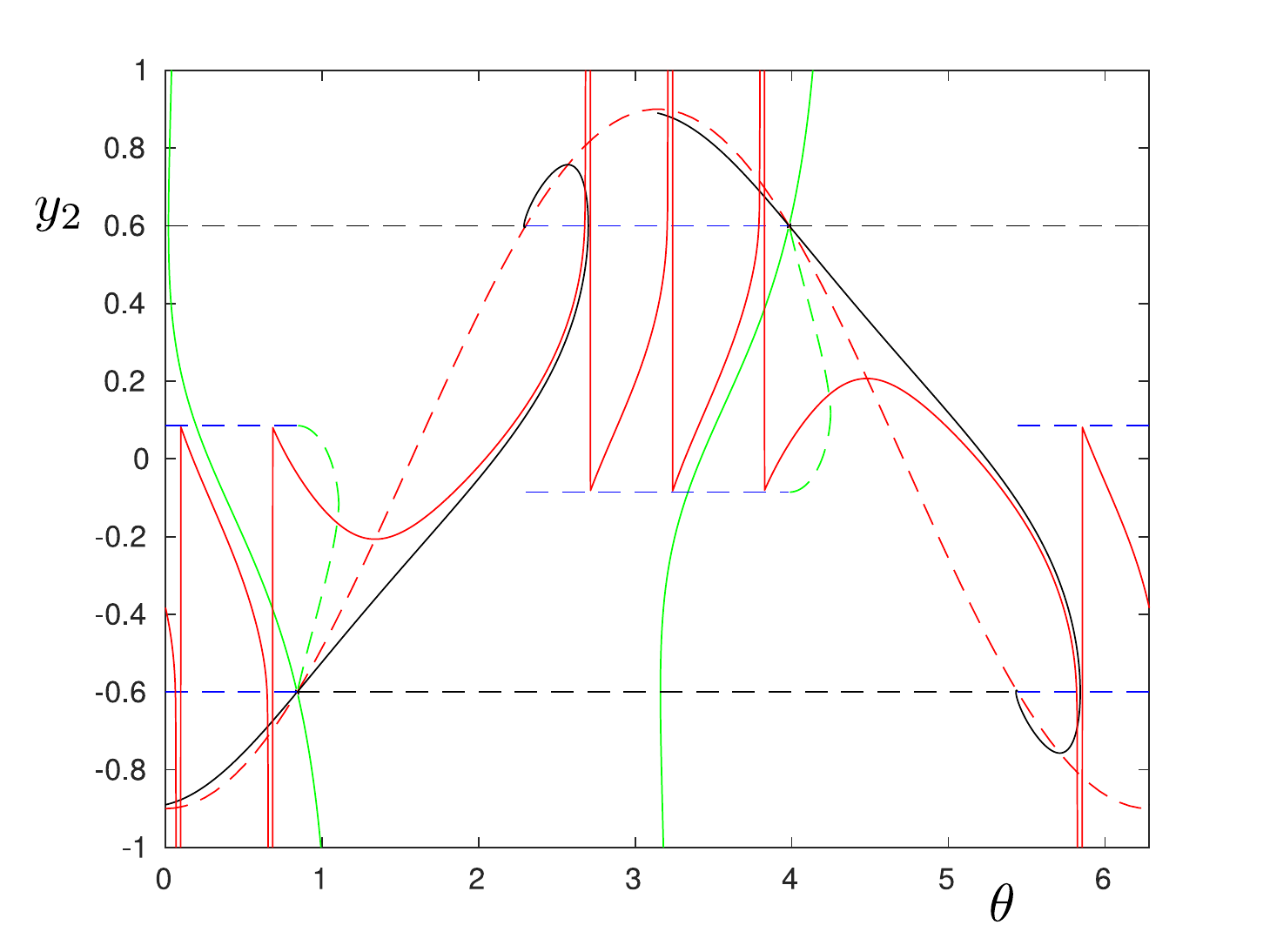}\end{subfigure}
% \begin{subfigure}{.45\textwidth}\caption{ }\label{fig:coexistence_cAndss3d}
%   \centering
%  \includegraphics[scale=0.215]{./coexistence_cAndss_ytheta.pdf}\end{subfigure}
 \begin{subfigure}{.45\textwidth}\caption{ }\label{fig:complicatedxvst}
  \centering
 \includegraphics[scale=0.415]{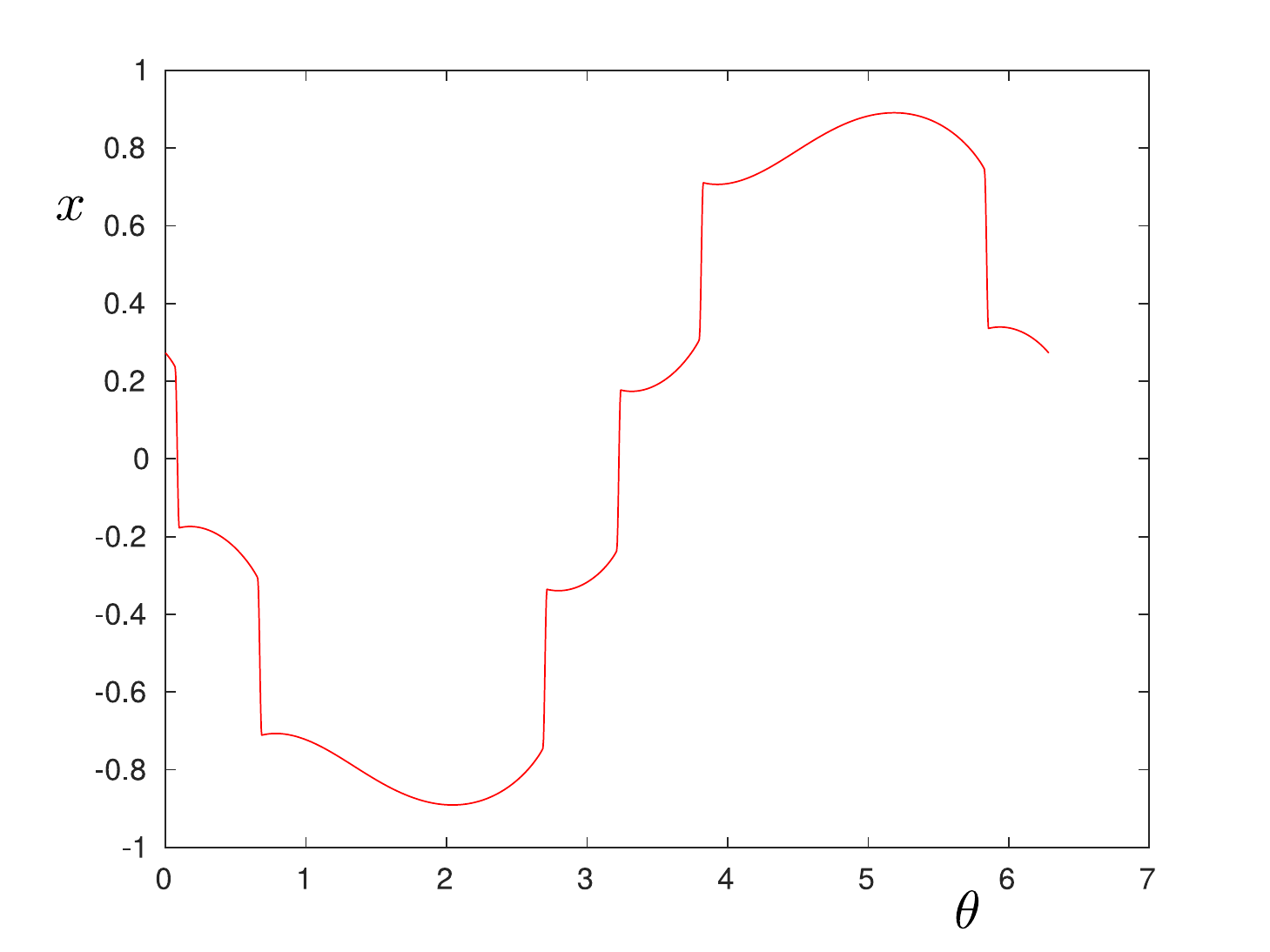}\end{subfigure}
  \begin{subfigure}{.45\textwidth}\caption{ }\label{fig:complicatedytheta_xi0_61}
  \centering
 \includegraphics[scale=0.415]{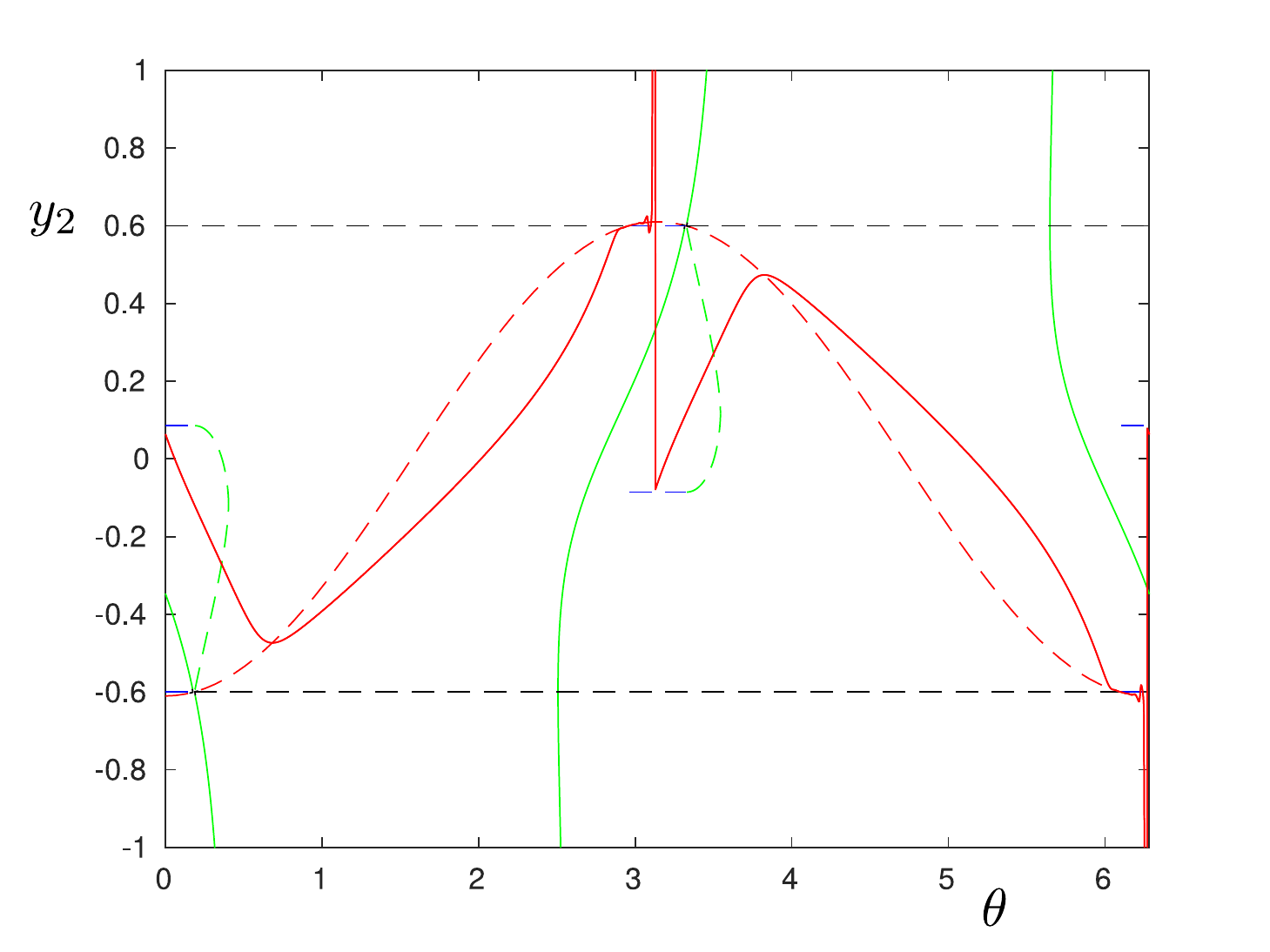}\end{subfigure}
   \caption{Same as in \Cref{fig:coexistence} but for $\mu_d=0.11$, $\mu_s=0.3$. In (a) and (b) we have $\xi=0.9$ whereas $\xi=0.61$ (just slightly above the critical value $\delta=0.6$) in (c). In (a), we see that $\gamma_+$ intersects $F_-$ and not $G_-(J_-)$ and as a result no pure-stick orbits exists. We see that the stick-slip periodic orbits are more complicated with several excursions into $y\gtrless \pm -c$ for $c>0$ sufficiently small for all $0<\vare\ll 1$. Specifically, in (a) the red stick-slip orbit has four such fast excursions. Figure (b) shows the corresponding $x(\theta)$. Finally, in (c) we see that the limit cycle in red passes through the folded nodes and consequently it is of mixed-mode type with small oscillations in this region. }\label{fig:complicated}
\end{figure}

\section{Canard-induced chaos}\label{sec:chaos}

In this final section, we prove existence of chaos through a simple geometric mechanism based upon the folded saddle that produces a horseshoe. The basic geometry is shown in \Cref{fig:reducedchaos}, which under assumption \ref{it:A5} occurs for $\xi\in (\xi_t,\xi_{pd})$, see also \Cref{fig:xitpdph} for the regularization function \cref{eq:phie} and parameter values \cref{eq:parahere}. 
In the case illustrated in \Cref{fig:reducedchaos}, $\gamma_+$ transversally intersects $G_-(\tilde \gamma_-)$ twice and by \Cref{thm:main1new} there are two limit cycles $\Gamma_{ssc}(\xi,\vare)$ and $\Gamma_{c}(\xi,\vare)$ with canard segments. (These coexists with the pure-stick orbit $\Gamma_{ps}(\xi,\vare)$ but this will play little role in the following. ) In principle, one intersection could be due to $L_-(\tilde \gamma_-)$, without essential changes to our approach, but for simplicity we will focus on $G_-(\tilde \gamma_-)$ here (also to avoid the generic case in-betwen $L_-(\tilde \gamma_-)$ and $G_-(\tilde \gamma_-)$.
% Consequently, there is only a symmetric pure-stick periodic orbit for all $0<\vare\ll 1$ according to \Cref{thm:main1new} in this case, but due to the transverse intersections of $\mathbb S\gamma_-$ and $G_-(\tilde \gamma_-)$ this periodic orbit coexist with additional ones with canard segments, see \Cref{fig:reducedchaos}. 
In particular, there are two singular periodic orbits $\gamma_{1}=\lim_{\vare\rightarrow 0} \Gamma_{ssc}(\xi,\vare)$ and $\gamma_2=\lim_{\vare\rightarrow 0} \Gamma_{c}(\xi,\vare)$, one for each transverse intersection, for $\vare=0$. Notice that each $\gamma_i$ on $C$ is given by two symmetric copies, one of which we parameterize by $\theta$ as follows
\begin{align*}
 \gamma_i:\,y_2=m(\theta),\quad \theta\in (\theta_i-\pi,\theta_{i}),
\end{align*}
where $\theta_i$ is the value of $\theta$ at the ``jump point'' on $C_r$. The function $m$ is smooth if $\gamma_i$ does not reach the the fold on the interval $\theta\in (\theta_i-\pi,\theta_i)$. Otherwise, by the mapping $G_+$, it becomes piecewise smooth. Notice that $\theta_2>\theta_1$ consisting with the definition of $\Gamma_{c}$ and $\Gamma_{ssc}$, see also \Cref{fig:lcs}.
% \begin{remark}
% \end{remark}

% \begin{itemize}
%  \item 
% Consider a canard segment $\gamma_\vare:=\{(x(\tau),y_2(\tau),\theta(\tau)):\,\tau\in [0,T]\}$ of \cref{eq:}, connecting $(x(0),y_2(0),\theta(0))\in S_{a,\vare}$ with 
% $(x(T),y_2(T),\theta(T))\in S_{r,\vare}$, and let $q_\vare$ be a point on the stable fiber of $(x(0),y_2(0),\theta(0))\in S_{a,\vare}$. Then 
% \begin{align*}
%   \exp\left(\int_{0}^T \vare^{-2} \lambda(y_2(\tau)) d\tau\right)
% \end{align*}
% gives a leading order estimate of the motion of $q_{\vare}$ normal to $\gamma_\vare$ $(x(T),y_2(T),\theta(T))$. Similarly, if $\gamma_\vare$ is a canard orbit then 
Consider a canard segment $\gamma_\vare:=\{(x(\tau),y_2(\tau),\theta(\tau)):\,\tau\in [0,T]\}$ of \cref{eq:scaling2}, connecting $(x(0),y_2(0),\theta(0))\in S_{a,\vare}$ with 
$(x(T),y_2(T),\theta(T))\in S_{r,\vare}$. Then the variational equations along $\gamma_\vare$ have a solution $(x',y_2',\theta')$ with 
\begin{align}
  y_2'\sim \exp\left(\int_{0}^T \vare^{-2} \lambda(y_2(\tau)) d\tau\right)\label{eq:y2p}
\end{align}
where $\lambda$ is defined in \cref{eq:lambda}. 
In particular, if $\int_{0}^T \lambda(y_2) d\tau = \xi^{-1} \int_{\theta_0}^{\theta_T} \lambda (y_2) d\theta<0$ with $\theta_0=\theta(0),\theta_T=\theta(T)$, then the contraction gained along $S_{a,\vare}$ dominates the the expansion along $S_{r,\vare}$ and the expression \cref{eq:y2p} is exponentially small. 
% gives a leading order estimate of the motion of $q_{\vare}$ normal to $\gamma_\vare$ $(x(T),y_2(T),\theta(T))$. Similarly, if $\gamma_\vare$ is a canard orbit then 
% Following on from this we then suppose the following:
Following on from this we then ...
\begin{enumerate}[resume*]
\item \label{it:A6} suppose that 
 \begin{align}\label{eq:lambdacond}
  \int_{\theta_1-\pi}^{\theta_2} \lambda(m(s)) ds <0.
 \end{align}
 % \end{itemize}
\end{enumerate}
 The interpretation of this condition is then that the contraction gained along the fast fibers of $\gamma_2$ from $\theta=\theta_1-\pi$ on the attracting side dominates the expansion on the repelling side up until $\theta_2$. Seeing that $\theta_2>\theta_1$ this condition ensures that the same holds for $\gamma_1$ from $\theta_1-\pi$ to $\theta_1$. We could also replace $<$ by $>0$ in \cref{eq:lambdacond}. We would then just consider the backward flow. For simplicity we stick to $<0$. 

\begin{figure}[!t]
        \centering
        \includegraphics[width=0.4\textwidth]{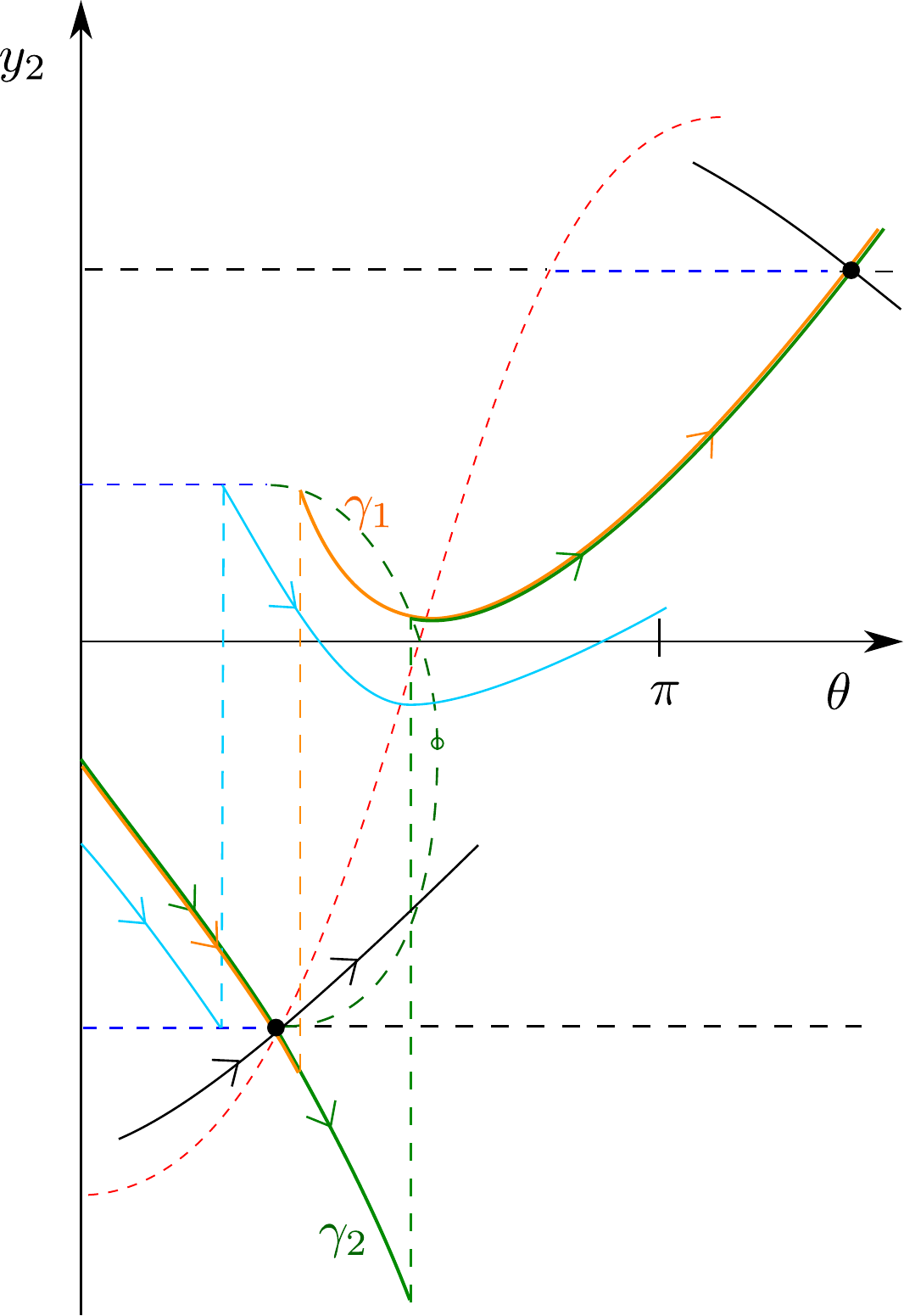}
        \caption{An illustration of a situation where the assumptions of \Cref{thm:main2} are satisfied.  }
        \label{fig:reducedchaos}
\end{figure}
% Suppose that 
% In this section, we give psufficient conditions for existence of a horseshoe. These conditions are described by the canards and the return  

% show the following:

% Let $P_\vare$ denote the stroposcopic mapping defined by \cref{eq:vectorfield_regularized} from $\{\theta=0\}$ to itself. Then if $Q_\vare$ denotes the mapping from $\{\theta=0\}$ to $\{\theta=\pi\}$ under the forward flow, we have $$P_\vare = (S\circ Q_\vare)^2,$$  
% and the following holds for $S\circ Q_\vare$:
\begin{theorem} \label{thm:main2}
Suppose \ref{it:A1}-\ref{it:A6} all hold and let $\xi\in (\xi_t,\xi_{pd})$ be so that $\gamma_+=\mathbb S\gamma_-$ transversally intersects $G_-(\tilde \gamma_-)$ twice as illustrated in \Cref{fig:reducedchaos}. Then there exists an $\vare_0>0$ small enough such that for all $\vare\in (0, \vare_0)$, $R_\vare$ has an invariant cantor-set $\Lambda_\vare$ where the restricted mapping  $R_\vare\vert_{\Lambda_\vare}$ is topologically conjugated to a shift of two symbols. 
\end{theorem}
\begin{proof}
 As before we fix $S_{a,\vare}$ and $S_{r,\vare}$ and let $\gamma_{-,\vare}$ be a perturbation of $\gamma_-$ connecting these manifolds \cite{szmolyan2001a}. We take $S_{a,\vare}$ to be symmetric \cite{haragus2011a} so that $\gamma_{+,\vare}=\mathbb S\gamma_{-,\vare}$ also belongs to $S_{a,\vare}$. We now redefine $\theta_*\gtrsim\theta_1-\pi$ in such way, that by the Exchange Lemma \cite{schecter2008a}, upon flowing a small neighborhood of $\gamma_{-,\vare}\cap \{\theta=\theta_*\}$ on $S_{a,\vare}$ forward, we obtain a(n) (specific) unstable manifold $W^u(\gamma_{-,\vare})$ of $\gamma_{-,\vare}$ on the repelling side. This set reaches, by following $G_-$, a vicinity of $\gamma_{+,\vare}$ on $S_{a,\vare}$ by assumption. By the stable foliation of fibers of $S_{a,\vare}$ we can therefore extend $W^u(\gamma_{-,\vare})$ in this way so that it is now foliated by stable fibers with base points on $S_{a,\vare}$. By transversality of $G_-(\tilde \gamma_{-})$ with $\gamma_{+}$ on $S_a$ in two points, as assumed, it follows that the foliation of $W^u(\gamma_{-,\vare})$ intersects $\gamma_{+,\vare}$ transversally in two points $q_\vare$ and $p_\vare$ on $S_{a,\vare}$. We consider two sufficiently small ``intervals'' $I_\vare$ and $J_\vare$ of base points of $W^u(\gamma_{-,\vare})$ that act as disjoint neighborhoods of $q_{\vare}$ and $p_{\vare}$ on $S_{a,\vare}$. Upon restricting $I_\vare$ and $J_\vare$ further if necessary the forward flow of these points on $S_{a,\vare}$ coincide as $S_{a,\vare}\cap \{\theta=\pi+\theta_*\}$ on a sufficiently small neighborhood of $\gamma_{+,\vare}\cap \{\theta=\pi+\theta_*\}$. Therefore, on this neighborhood $W^u(\gamma_{-,\vare})$ becomes two one-dimensional disjoint ``strips'' $H_{1,\vare}$ and $H_{2,\vare}$ that are exponentially close to $S_{a,\vare}\cap \{\theta=\pi+\theta_*\}$. By construction, the ``preimages'' on $\theta=\theta_*$ of these strips are exponentially small intervals $K_\vare$ and $M_\vare$ on $S_{a,\vare}\cap \{\theta=\theta_*\}$, both exponentially close to $\gamma_{a,\vare}$. 
 
%  We now redefine the map $P_\vare$ as a mapping obtained from $\{\theta=\theta_*\}$ to $\{\theta=\pi+\theta_*\}$ by the forward flow and let $R_\vare$ be such that $P_\vare = R_\vare^2$ using the symmetry $\mathbb S$. It is $R_\vare$ that we study in the following.
 
 Now, consider the stable manifolds $V_{1,\vare}:=W^s(K_\vare)$ and $V_{2,\vare}:= W^s(M_\vare)$ of $K_\vare$ and $M_\vare$, respectively, within $\{\theta=\theta_*\}$. %In fact, we will restrict $V_{i,\vare}$ to the $y_2$-positive side only. 
 These sets are then each $\mathcal O(1)$ in the ``verticle'' direction, transverse to $S_{a}$, and exponentially thin in the ``horizontal'' direction tangent to $S_a$. By \cref{eq:lambdacond}, the forward flow of these sets produce exponential thickenings of the horizontal  strips $H_{1,\vare}$ and $H_{2,\vare}$, being the images of $K_\vare$ and $M_\vare$, respectively. We continue to denote these ``thickened'' objects by the same symbols $H_{1,\vare}$ and $H_{2,\vare}$. It then follows that $\mathbb S H_{1,\vare}$ and $ \mathbb S H_{2,\vare}$ are two horizontal strips at $\theta=\theta_*$ that intersect $V_{1,\vare}$ and $V_{2,\vare}$ in four exponentially small rectangles. This gives the desired horseshoe mechanism. We then proceed to verify the cone properties of the Conley-Moser conditions using \cite[Theorem 25.2.1]{wig}.  For this, we again use the slow-fast structure and define an ``expanding cone'' $K_u$ of opening angle $d_u>0$ (small enough) centered at tangent vectors $(x', y_2')$ of $C_a$ at $\gamma_- \cap \{\theta=\theta_*\}$. Indeed, consider then a point $q\in V_{i,\vare}$ and let $p=R_{\vare}(q)\in H_{i,\vare}$. Then, since the intersection of $S_{a,\vare}$ and $S_{r,\vare}$ is transverse, and since $G_-(\tilde \gamma_-)$ is transverse to the slow flow on $C_a$, it follows that $K_u$ is invariant for $DR_{\vare}$ and vectors is expanding exponentially due the motion near $C_r$. In summary, $DR_{\vare}(q)(K_u)\subset K_u$ and 
 \begin{align*}
  \vert DR_{\vare}(q)(v) \vert \ge e^{c_u/\vare} \vert v\vert,\quad v\in K_u, 
 \end{align*}
for some $c_u>0$ and all $0<\vare\ll 1$. Similarly, we define a ``contracting cone'' $K_s$ of opening angle $d_s>0$ (small enough) centered at tangent vectors $(0,y_2')$ of the leaves of the stable foliation of $C_a\cap \{\theta=\theta_*\}$. Then by the Exchange Lemma and the transverse intersection of $S_{a,\vare}$ and $S_{r,\vare}$, we have that $DR_{\vare}^{-1}(p)(K_s)\subset K_s$ and 
 \begin{align*}
  \vert DR_{\vare}^{-1}(p)(v) \vert \ge e^{c_s/\vare} \vert v\vert,\quad v\in K_s, 
 \end{align*}
 for some $c_s>0$ and all $0<\vare \ll 1$.
%  The stable foliation of $S_{a,\varepsilon}$ gives rise to stable foliations of $H_{i,\vare}$, $i=1,2$. Flowing the leaves of this foliation backwards, we obtain a foliation of $V_{i,\vare}$; notice in particular, that these sets belong to $V_{i,\vare}$ since they are preimages of points in $H_{i,\vare}$, $i=1,2$. 
%  
%  Along the leaves of this foliation of $V_{i,\vare}$, we then have, by construction, an exponential contraction. In fact, at each point in $V_{1,\vare}$ there is a (very) small cone centered along these leaves, which is invariant for $DR_\epsilon$, see also \cite{haiduc2009a,kuehn2015a}, within which vectors are exponentially contracted. On the other hand, within a fixed transverse direction to these leaves,  there is, due to the transverse intersection of $S_{a,\varepsilon}$ and $S_{r,\varepsilon}$, an exponential expansion. This also gives rise to a disjoint cone, being invariant for $DR_\epsilon$, within which vectors are exponentially expanded. The Conley-Moser conditions are therefore verified and 
Having established the cone properties, the result therefore follows from classical results on Smale's horseshoe, see e.g. \cite[Theorem 25.1.5]{wig}.  
%  \note{Can we fit the cones in the sets $V_{i,\vare}$?}
%  Using $G_-$,  $S_{a,\vare}$ with base points on $\gamma_{-,\vare}$ on the repelling side. The image of this ident
\end{proof}

% We emphasize that the conditions for \Cref{thm:main2} are numerically verified for the parameter values in \cref{eq:parahere}, $\mu_d=0.4$ and for any $\xi$ between the value $\xi \approx 0.7805$ in \Cref{fig:coexistence_csn_ytheta}, where the tangency occurs, and the value $\xi\approx 0.8096$ in  \Cref{fig:coexistence_csn_ytheta}, where $\gamma_+$ intersects $G_-(z_-)$. Although this verification is based upon numerical computations, this can in principle be done to any desired accuracy, seeing that it just involves the reduced problem and the mapping $G_-$. 
% REWRTE: The cantor-set is obtained from a horseshoe and the main message of this paper is to outline a new simple mechanism based upon a combination a folded saddle singularity and a simple return mapping. The folded saddle is also the main object used in \cite{} to prove the existence of a horseshoe in the forced van der Pol. But the main mechanism in this paper is simpler and the onset of chaos is determined by a simple tangency conditions that essentially boils down to local computation. 
% 
% Next, suppose instead that:
% 
% \textbf{Assumption 2}. $R(\tilde \gamma_0)$ does not intersect $\mathbb S\gamma_0$. 
% 
% In this case, we have:
% % If assumption 1
% \begin{theorem} \label{thm:main2}
% Suppose Assumption 2. Then the global attractor of the system consists entirely of an attracting limit cycle on the slow manifold $S_{a,\vare}$???
% \end{theorem}
% 
% \section{A GSPT approach}

% \section{
% we study this case in further details, setting 
% \begin{align*}
%  \gamma = 
% \end{align*}

 \section{Discussion}\label{sec:disc}
In this paper, we have solved some open problems on the bifurcation of stick-slip limit cycles for a regularized model \cref{eq:model0} of a spring-mass-friction oscillator in the limit where the ratio $\omega$ of the forcing frequency and the natural spring frequency is comparable with the scale associated with the friction. In particular, using GSPT and blowup we showed that the existence (and nonexistence) of limit cycles is directly related to a folded saddle and the location of the canard relative to the image of the fold lines under a return mechanism to the attracting critical manifold, see also \Cref{thm:main1new}. 

There are parameter regimes with $\mu_s<1$ that we did not pursue rigorously. However, in \Cref{rem:complicated} we pointed out that numerical computations in this regime also reveal interesting dynamics. In particular, it appears that there through the canards are spike-adding bifurcations \cite{carter2020a} in this case where a ``spike'' in our friction context is a slip phase. 

Finally, in \Cref{sec:chaos} we showed that the fold bifurcation of canard-type limit cycles is associated with the existence of a horseshoe, see \Cref{thm:main2}, and chaotic dynamics. The horseshoe mechanism is comparable to the mechanism for chaos in the forced van der Pol \cite{haiduc2009a}, involving a simple interplay between the folded saddle singularity and the return mechanism. In any case, it provides a simple way to prove existence of chaos in systems of the form \cref{eq:nonsm}. 

In terms of the friction modelling, it is interesting to note that the dynamics in the limit $\omega\rightarrow 0$ depends upon the details (i.e. the function $\phi$) of the friction force. For $\omega = \mathcal O(1)$, we know from \cite{bossolini2017b}, and also from more abstract results on the connection between Filippov systems \cite{filippov1988a} and the regularization of piecewise smooth systems \cite{kristiansen2014a,llibre2008a}, that the dynamics is qualitatively independent of $\phi$. Based upon this observation, we are led to the conclusion that the details of friction is to be determined on ``diverging time scales''. Although it is well established that our friction model (depending solely on the relative velocity) is over simplified \cite{woodhouse2015a}, this conclusion is consistent with experimental results \cite{heslot} as well as with friction modelling in regimes of low and high relative velocity, see e.g. \cite{putelat2010a,woodhouse2015a}. 

% the bifurcation of limit cycles is directly related canards.  

\end{document}